\documentclass[11pt]{amsart}
\usepackage{amsmath}
\usepackage{amsthm}
\usepackage{amsfonts}
\usepackage{amssymb}
\usepackage{amscd}
\numberwithin{equation}{section}
\newtheorem{theorem}{Theorem}[section]
\newtheorem{cor}[theorem]{Corollary}

\newtheorem{proposition}[theorem]{Proposition}
\newtheorem{lemma}[theorem]{Lemma}
\newtheorem{prop}[theorem]{Proposition}
\theoremstyle{definition}
\newtheorem{defi}{Definition}[section]
\newtheorem{rem}[defi]{Remark}

\newtheorem{example}[defi]{Example}

\newcommand\half{\frac{1}{2}}

\newcommand\be{\beta}

\newcommand\g{\mathfrak g}

\newcommand\h{\mathfrak h}

\newcommand\D{\Delta}
\renewcommand\l{\lambda}

\renewcommand\d{\delta}
\newcommand \Dim{\rm Dim}

\renewcommand\a{\alpha}
\renewcommand\aa{\mathfrak a}

\newcommand{\Z}{\mathbb Z}
\renewcommand\th{\theta}

\newcommand\ganz{\mathbb Z}

\renewcommand\aa{\mathfrak a}

\newcommand\C{\mathbb C}
\newcommand\R{\mathbb R}

\newcommand{\sdim}{\text{\rm sdim}}

\newcommand{\vac}{{\bf 1}}
\newcommand{\bea}{\begin{eqnarray}}
\newcommand{\eea}{\end{eqnarray}}
\begin{document}
\title[Conformal embeddings in  $W$-algebras I]{Conformal embeddings of affine vertex algebras in minimal $W$-algebras I:   structural results}
\author[Adamovi\'c, Kac, M\"oseneder, Papi, Per\v{s}e]{Dra{\v z}en~Adamovi\'c}
\author[]{Victor~G. Kac}
\author[]{Pierluigi~M\"oseneder Frajria}
\author[]{Paolo~Papi}
\author[]{Ozren~Per\v{s}e}
\begin{abstract} We find all values of $k\in \C$, for which the embedding of the maximal affine vertex algebra in a simple 
minimal W-algebra $W_k(\g,\theta)$ is conformal, where $\g$ is a basic simple Lie superalgebra
and $-\theta$ its minimal root.
In particular, it turns out that if $W_k(\g,\theta)$ does not collapse to its affine part, 
then the possible values of these $k$ are either $-\frac{2}{3}Êh^\vee$ or $-\frac{h^\vee-1}{2}$, where 
$h^\vee$ is the dual Coxeter number of $\g$ for the normalization 
$(\theta,\theta)=2$.  As an application of our results, we present a realization of simple affine vertex algebra $V_{-\tfrac{n+1}{2} } (sl(n+1))$ inside of the tensor product of  the vertex algebra $W_{\tfrac{n-1}{2}} (sl(2\vert n), \theta)$  (also called the Bershadsky-Knizhnik algebra) with a lattice vertex algebra.
\end{abstract}
\maketitle
\centerline{\it To Efim Zelmanov, on the occasion of his 60th birthday}
\section{Introduction}\label{zero}
Let $V$ be a vertex algebra with a Virasoro (= conformal)  vector $\omega_V$ and let $W$ be a vertex subalgebra of $V$ endowed with a Virasoro vector  $\omega_W$. The embedding $W\subset V$ is called {\it conformal} if $\omega_W=\omega_V$ \cite{AKMPP}.\par
 This notion is a generalization of the notion of conformal embedding of a reductive Lie algebra $\g^0$ in a simple Lie algebra $\g$, studied extensively in the mathematics and physics literature in the 1980's, due to its relevance to string compactifications.
Namely, the embedding $\g^0\subset \g$ was called conformal,  if there exists a nontrivial vacuum integrable highest weight module over the affine Lie algebra $\widehat \g$,  such that the Sugawara constructions of  the Virasoro algebra for
$\widehat \g$ and for $\widehat{\g^0}$ coincide. It is not difficult to show that this property  is equivalent to 1) the equality of Sugawara central charges, and also equivalent to 2) the finiteness of the decomposition of  each eigenspace of the center of $\g^0$ of an irreducible level $1$ highest weight $\widehat \g$-module (equivalently, all of them) with respect to $\widehat{\g^0}$. Maximal such conformal embeddings were classified in \cite{SW} (see also  \cite{AGO}), and the corresponding decompositions of level  $1$  $\widehat \g$-modules with respect to $\widehat{\g^0}$ were completely described in \cite{KWW}, \cite{KS}, \cite{CKMP}, \cite{KMPX}. 

More recently, some new examples of conformal embeddings for affine vertex algebras have been found \cite{A}, \cite{AP2}, \cite{Perse}. Of course, the
corresponding $\widehat \g$-modules are not integrable, but still the equality of Sugawara conformal vectors turned out to be equivalent to the equality of their central charges. Moreover, it turned out that the decompositions of these modules (when known) are always completely reducible, though not necessarily finite (at each eigenspace of the genter of $\g^0$). In almost all cases they are described in \cite{AKMPP}.\par
Recall that for an affine vertex (super)algebra $V$ of level $k$, associated to a simple Lie (super)algebra $\g$ with a (super)symmetric non-zero invariant bilinear form, the Sugavara Virasoro vector is given by the formula $$\omega_V=\frac{1}{2(k+h^\vee)}\sum_i b_i a_i, $$ where $\{a_i\},\,\{b_i\}$ are dual bases of $\g$, $h^\vee$ is the dual Coxeter number (i.e., half  of the eigenvalue of the Casimir element of $\g$), and $k+h^\vee$ is assumed to be nonzero.\par
In the present paper we study conformal embeddings of the maximal affine vertex subalgebra of a simple minimal W-algebra $W_k(\g,\theta)$, attached to a minimal grading of a simple basic Lie superalgebra $\g$. Recall that a {\it} minimal grading of $\g$ is the 
$\frac{1}{2}\ganz$-grading of the form 
$$
\g=\g_{-1}\oplus\g_{-1/2}\oplus\g_{0}\oplus\g_{1/2}\oplus\g_{1},
$$
where $\g_{\pm 1}=\C  e_{\pm \theta}$   are the root spaces of $\g$ attached to the even roots $\pm \theta$. It is well known that there exists a unique up to isomorphism such grading for any simple Lie algebra, but for a simple basic Lie superalgebra there are several (see \cite{KW}).\par
It is proven in \cite{KW}Ê that, for $k\ne -\ h^\vee$, the universal minimal W-algebra $W^k(\g,\theta)$ has a unique simple quotient, denoted by $W_k(\g,\theta)$, and that $W^k(\g,\theta)$
is freely generated by the elements  $J^{\{a\}}$ ($a$ runs over a basis of $\g^\natural$), $G^{\{u\}}$ ($u$ runs over a basis of $\g_{-1/2}$), and the Virasoro vector $\omega$, where $\g^\natural$ is the centralizer in $\g_0$ of $e_{-\theta}$ (hence of $e_\theta$). Furthermore the elements $J^{\{a\}}$ (resp. $G^{\{u\}}$) are primary of conformal weight $1$ (resp. $3/2$), with respect  to $\omega$, and they satisfy the following $\l$-brackets:
\begin{align*}
[{J^{\{a\}}}_{\lambda}J^{\{b\}}]&=J^{\{[a,b]\}}+\lambda\left((k+h^\vee/2)(a|b)-\tfrac{1}{4}\kappa_0(a,b)\right),\ a,b\in \g^\natural,\\
[{J^{\{a\}}}_{\lambda}G^{\{u\}}]&=G^{\{[a,u]\}},\ a\in \g^\natural,\,u\in\g_{-1/2},
\end{align*}
where $(\cdot | \cdot)$ is the invariant bilinear form of $\g$, normalized by the condition $(\theta|\theta)=2$ and $\kappa_0$ is the Killing form of $\g_0$. The  $\l$-brackets between elements $G^{\{u\}}$ and $G^{\{v\}}$  given in \cite{KW2} Ê(see \eqref{GG}) are much more complicated.\par
Our first key result is an important simplification of this formula. It reads (see Lemma \ref{formG}):
\begin{align}\label{GGsimplified}
&[{G^{\{u\}}}_{\lambda}G^{\{v\}}]=-2(k+h^\vee)(e_\theta|[u,v])\omega+(e_\theta|[u,v])\sum_{\alpha=1}^{\dim \g^\natural} 
:J^{\{a^\alpha\}}J^{\{a_\alpha\}}:+\\\notag
&\sum_{\gamma=1}^{\dim\g_{1/2}}:J^{\{[u,u^{\gamma}]^\natural\}}J^{\{[u_\gamma,v]^\natural\}}:
+2(k+1)\partial J^{\{[[e_\theta,u],v]^\natural\}}\\\notag
&+ 4 \l  \sum_{i\in I} \frac{p(k)}{k_i} J^{\{[[e_\theta, u],v]_i^\natural\}}+2\l^2(e_\theta|[u,v])p(k)\vac.
\end{align}
Here $\{a_\alpha\}$ (resp. $\{u_\gamma\}$) is a basis of  $\g^\natural$ (resp. $\g_{1/2}$) and $\{a^\alpha\}$ (resp. $\{u^\gamma\}$) is the corresponding dual basis w.r.t. $(\cdot|\cdot)$ (resp w.r.t. $\langle\cdot,\cdot\rangle_{\rm ne}=(e_{-\theta}|[\cdot,\cdot])$),  $a^\natural$ is the  orthogonal projection of $a\in\g_{0}$ on $\g^\natural$, $a_i^\natural$ is the projection of $a^\natural$ on the $i$th minimal ideal $\g_i^\natural$ of $\g^\natural$, $k_i=k+\frac{1}{2}(h^\vee-h^\vee_{0,i})$, where $h^\vee_{0,i}$ is the dual Coxeter number of $\g_i^\natural$ with respect to the restriction of the form $(\cdot|\cdot)$.\par
The important new ingredient here is the monic quadratic polynomial $p(k)$, which is given in Table 4 of the present paper. (We show that all linear polynomials $k_i$ divide $p(k)$).
Note that in all cases $\g^\natural$ is a direct sum of all its minimal ideals, with the exception of $\g= sl(2+n|n)\ (n>0)$, which we exclude from consideration.

This result allows us to determine all non--critical {\sl collapsing levels}, i.e. those $k$ with $k\ne-h^\vee$ for which the image in $W_k(\g,\theta)$ of the universal affine vertex algebra generated by $J^{\{a\}}$, $a\in\g^\natural$, coincides with the whole vertex algebra $W_k(\g,\theta)$ (see Section \ref{cl}). Namely, we show that a non-critical level $k$ is collapsing iff $p(k)=0$ (Theorem \ref{TT}).\par

Special collapsing levels are those $k$
 for which $W_k(\g,\theta)=\C\vac$. It turns out that this happens precisely in the following two situations (see Proposition \ref{ArakawaMoreau}):  
\begin{enumerate}
\item $k=-\frac{h^\vee}{6}-1$ and $\g$ is one of the Lie algebras of exceptional Deligne's series \cite{D}: $A_2$, $G_2$, $D_4$, $F_4$, $E_6$, $E_7$, $E_8$, or $\g=psl(m|m)$ ($m\ge2$), $osp(n+8|8)$ ($n\ge2$), $spo(2|1)$, $F(4)$, $G(3)$ (for both choices of $\theta$); in all these cases
$$
p(k)=(k+\frac{h^\vee}{6}+1)(k+\frac{h^\vee}{3}).
$$
\item $k=-\frac{1}{2}$ and $\g=spo(n|m)$ ($n\ge1$); in this case
$$
p(k)=(k+\half)(k+\frac{h^\vee+1}{6}).
$$
\end{enumerate}

Note that a connection of minimal simple $W$-algebras to Deligne's series was first observed in \cite{Kawa}, and the collapsing levels $k$ for which $W_k(\g,\theta)=\C\vac$ when $\g$ is Lie algebra have been already found in \cite{AM}. Note also that, as an application, we obtain an (almost) unified proof of the formula for the superdimension when $\g$ appears in (1) above (cf. \cite{D}, \cite{Wes}):
$$
\sdim\, \g=\frac{2(h^\vee+1)(5h^\vee-6)}{h^\vee+6}.
$$

We then proceed to determine the {\sl conformal levels}, i.e.\  those $k$ for which the image in $W_k(\g,\theta)$ of the universal affine vertex algebra generated by $J^{\{a\}}$, $a\in\g^\natural$, is conformally embedded in $W_k(\g,\theta)$. Obviously, the collapsing levels are conformal. A remarkable fact, discovered in this paper, is that a non-collapsing conformal embedding may happen only if $k=-\frac{2}{3}h^\vee$ or $k=\frac{1-h^\vee}{2}$. We do not have a uniform explanation for this; we discovered this fact after listing all non-collapsing conformal levels by a case-wise verification (see Section \ref{section4}).
Another striking fact follows from this: the equality of central charges of the affine part and of $W_k(\g,\theta)$ is not only (obviously) a necessary, but also a sufficient condition for $k$ to be a conformal level! 

A further application of our simplified $\l$-bracket is given in Section \ref{sln+1inW}. It is the generalization to higher rank of the realization,  discovered in \cite{A-2014}, of  $\widehat{sl(3)}$ at level $-\frac{3}{2}$ in the $N=4$ superconformal algebra tensored with a lattice vertex algebra. More precisely we realize $\widehat{sl(n+1)}$ at level $-\frac{n+1}{2}$ as the zero charge component of the tensor product of $W_{\frac{n-1}{2}}(sl(2|n),\theta)$ with a rank one lattice vertex algebra.  
The $W$-algebras ${W}_{k}(sl(2|n), \theta)$  have been very much investigated in physics literature under the name Bershadsky-Knizhnik superconformal algebra \cite{Ber},  \cite{IMP}, \cite{Kn}, \cite{Ras}.  Our result show that these algebras at certain central charges are closely related with  affine vertex algebras at level $-\frac{n}{2}$.

In our subsequent paper \cite{AKMMP2} we describe in most of the  cases how $W_k(\g,\theta)$ decomposes with respect to its affine part when $k$ is a conformal (non-collapsing) level. In all these cases the decomposition turns out to be completely reducible, though not necessarily finite.
\vskip10pt
{\bf Notation.} The base field is $\C$. As usual, tensor product of a  family of vector spaces ranging over the empty set is meant to be $\C$. For a vector superspace $V=V_{\bar 0}\oplus V_{\bar 1}$ we set   $\Dim V= \dim V_{\bar 0} | \dim V_{\bar 1}$,   $\dim V= \dim V_{\bar 0} + \dim V_{\bar 1}$    and $\sdim V= \dim V_{\bar 0} - \dim V_{\bar 1}$.

\vskip10pt

{\bf Acknowledgments.} We would like to thank T. Arakawa, I. Entova and J. Rasmussen for useful discussions on various subjects related to  this paper.\par
Dra{\v z}en~Adamovi\'c  and Ozren~Per\v{s}e are partially supported by the Croatian Science Foundation under the project 2634 and by the Croatian Scientific Centre of Excellence  QuantixLie.
Victor  Kac is partially supported by a NSF grant. Pierluigi M\"oseneder Frajria and Paolo Papi are partially supported by PRIN project ``Spazi di Moduli e Teoria di Lie''.

\section{Minimal quantum affine $W$--algebras}\label{uno}

Let $\g$ be a basic simple Lie superalgebra. Recall  that among all simple finite-dimensional Lie superalgebras it is characterized by the properties that its even part $\g_{\bar 0}$ is reductive and that it admits a non-degenerate invariant supersymmetric bilinear form $( \cdot | \cdot)$. A complete list of basic simple Lie superalgebras  consists of simple 
finite-dimensional Lie algebras and the Lie superalgebras $sl(m|n)\ (m,n\ge 1, m\ne n)$, $psl(m|m)\ (m\ge 2)$, $osp(m|n)=spo(n|m)\ (m\geq 1, n\ge 2 \text{ even})$, $D(2,1;a)\ (a\in\C,\,a\ne 0,-1)$, $F(4)$, $G(3)$ \cite{KacLie}. Recall that $sl(2|1)$and $spo(2|2)$ are isomorphic. Also, the Lie superalgebras $D(2,1;a)$ and $D(2,1;a')$ are isomorphic if and only if $a, a'$  lie on  the same orbit of the group generated by the transformations $a\mapsto a^{-1}$ and $a\mapsto -1-a$, and $D(2,1;1)= osp(4|2)$. See \cite{KacLie} for details.\par
Choose a Cartan subalgeba $\h$ for $\g_{\bar 0}$  and let $\D$ be the set of roots. A root $-\theta$ is called {\it minimal} if it is even and there exists an additive  function $\varphi:\D\to \R$ such that $\varphi_{|\D}\ne 0$ and $\varphi(\theta)>\varphi(\eta),\,\forall\,\eta\in\D\setminus\{\theta\}$. Obviously, a minimal root $-\theta$ is the lowest root of one of the simple components of $\g_{\bar 0}$ (in the ordering defined by $\varphi$). Conversely, it is easy to see, using the description of $\D$ given in \cite{KacLie}, that a long root of any simple component of $\g_{\bar 0}$ (with respect to its Killing form) is minimal except when $\g=osp(3|n)$ and the simple component of $\g_{\bar 0}$ is $so(3)$.\par
Fix a minimal root $-\theta$ of $\g$. We may choose root vectors $e_\theta$ and $e_{-\theta}$ such that 
$$[e_\theta, e_{-\theta}]=x\in \h,\qquad [x,e_{\pm \theta}]=\pm e_{\pm \theta}.$$
Due to the minimality of $-\theta$, the eigenspace decomposition of $ad\,x$ defines a {\it minimal} $\frac{1}{2}\ganz$-grading  (\cite[(5.1)]{KW}):
\begin{equation}\label{gradazione}
\g=\g_{-1}\oplus\g_{-1/2}\oplus\g_{0}\oplus\g_{1/2}\oplus\g_{1},
\end{equation}
where $\g_{\pm 1}=\C  e_{\pm \theta}$.   We thus have  a bijective correspondence between minimal gradings (up to an automorphism of $\g$) and minimal roots (up to the action of the Weyl group). Furthermore, one has
\begin{equation}\label{gnatural}
\g_0=\g^\natural\oplus \C x,\quad\g^\natural=\{a\in\g_0\mid (a|x)=0\}.
\end{equation}
Note that  $\g^\natural$ is the centralizer of the triple $\{f_\theta,x,e_\theta\}$.
We can choose $
\h^\natural=\{h\in\h\mid (h|x)=0\},
$ as a  Cartan subalgebra of the Lie superalgebra $\g^\natural$,  so that $\h=\h^\natural\oplus \C x$.\par
For a given choice of a minimal root $-\theta$, we normalize the invariant bilinear form $( \cdot | \cdot)$ on $\g$ by the condition
\begin{equation}\label{normalized}
(\theta | \theta)=2.
\end{equation}
The dual Coxeter number $h^\vee$ of the pair $(\g, \theta)$ (equivalently, of the minimal gradation \eqref{gradazione}) is defined to be   half the eigenvalue of the Casimir operator of $\g$ corresponding to $(\cdot|\cdot)$, normalized  by \eqref{normalized}.
\par
 The complete list of the Lie superalgebras $\g^\natural$, the $\g^\natural$--modules $\g_{\pm 1/2}$ (they are isomorphic and  self-dual),  and $h^\vee$ for all possible choices of $\g$ and of $\theta$ (up to isomorphism)  is given in Tables  1,2,3 of \cite{KW}. For the reader's convenience we reproduce them below. Note that in these tables 
 $\g=osp(m|n)$ (resp. $\g=spo(n|m)$) means that $\theta$ is the highest root of the simple component $so(m)$ (resp. $sp(n)$) of $\g_{\bar 0}$. Also, for $\g= sl(m|n)$ or $psl(m|m)$ we always take $\theta$ to be the highest root of the simple component $sl(m)$ of $\g_{\bar 0}$ (for $m=4$ we take one of the simple roots). Note that the exceptional Lie superalgebras $\g=F(4)$ and $\g=G(3)$ appear in both Tables  2 and 3, which corresponds to the two inequivalent choices of $\theta$, the first one being a root of the simple component $sl(2)$ of $\g_{\bar 0}$.
\vskip10pt
{\tiny
\centerline{Table 1}

\noindent {\sl $\g$ is a simple Lie algebra.}
\vskip 5pt
\centerline{\begin{tabular}{c|c|c|c||c|c|c|c}
$\g$&$\g^\natural$&$\g_{1/2}$&$h^\vee$&$\g$&$\g^\natural$&$\g_{1/2}$&$h^\vee$\\
\hline
$sl(n), n\geq 3$&$gl(n-2)$&$\C^{n-2}\oplus (\C^{n-2})^* $&$n$&$F_4$&$sp(6)$&$\bigwedge_0^3\C^6$ & $9$\\
$so(n), n\geq 5$&$sl(2)\oplus so(n-4)$&$\C^2\otimes\C^{n-4}$&$n-2$&$E_6$&$sl(6)$&$\bigwedge^3\C^6$ & $12$\\
$sp(n), n\geq 2$&$sp(n-2)$&$\C^{n-2} $&$n/2+1$&$E_7$&$so(12)$&$spin_{12}$ & $18$\\
$G_2$&$sl(2)$&$S^3\C^2$&$4$&$E_8$&$E_7$&$\dim=56$ & $30$\\
\end{tabular}}
\vskip 25pt
\centerline{Table 2}

\noindent {\sl $\g$ is not a Lie algebra but $\g^\natural$ is and $\g_{\pm1/2}$ is purely odd ($m\ge1$).}
\vskip 5pt
\centerline{\begin{tabular}{l|c|c|c||c|c|c|c}
$\g$&$\g^\natural$&$\g_{1/2}$&$h^\vee$&$\g$&$\g^\natural$&$\g_{1/2}$&$h^\vee$\\
\hline
$sl(2|m),$&$gl(m)$&$\C^{m}\oplus (\C^{m})^* $&$2-m$&$D(2,1;a)$&{\tiny $sl(2) \oplus sl(2)$}&$\C^2\otimes \C^2$ & $0$\\
$m\ge2$& & & & & & \\
$psl(2|2) $&$sl(2)$&$\C^2\oplus\C^{2}$&$0$&$F(4)$&$so(7)$&$spin_7$ & $-2$\\
$spo(2|m)$&$so(m)$&$\C^{m} $&$2-m/2$&$G(3)$&$G_2$&$\Dim= 0|7$ & $-3/2$\\
$osp(4|m)$&$sl(2)\oplus sp(m)$&$\C^2\otimes \C^m$&$2-m$\\
\end{tabular}}
\vskip 25pt
\centerline{Table 3}

\noindent {\sl Both $\g$ and $\g^\natural$ are  not  Lie algebras. ($m,n\geq 1$)}
\vskip 5pt
\centerline{\begin{tabular}{c|c|c|c}
$\g$&$\g^\natural$&$\g_{1/2}$&$h^\vee$\\
\hline
$sl(m|n)$, $n\ne m>2$&$gl(m-2|n)$&$\C^{m-2|n}\oplus(\C^{m-2|n})^*$&$m-n$\\
\hline
$psl(m|m),\,m>2$&$sl(m-2|m)$& $\C^{m-2|m}\oplus(\C^{m-2|m})^*$&$0$\\
\hline
$spo(n|m),\,n\ge 4$& $spo(n-2|m)$ &$\C^{n-2|m}$&$1/2(n-m)+1$\\
\hline
$osp(m|n),\,m\geq 5$&$osp(m-4|n)\oplus sl(2)$ &$\C^{m-4|n}\otimes \C^2$&$m-n-2$\\
\hline
 $F(4)$&$D(2,1;2)$ &$\Dim=6|4$& $3$\\
 \hline
$G(3)$&$osp(3|2)$ &$\Dim=4|4$& $2$\\
\end{tabular}}
}
\vskip20pt
In this paper we shall exclude the case of $\g=sl(n+2|n)$, $n>0$. In all other cases the Lie superalgebra $\g^\natural$ decomposes in a direct sum of all its minimal ideals, called components of $\g^\natural$:
$$\g^\natural=\bigoplus\limits_{i\in I}\g^\natural_i,
$$
where each  summand is either the (at most 1-dimensional)  center of $\g^\natural$ or is a basic classical simple Lie superalgebra different from $psl(n|n)$. We will also exclude $\g=sl(2)$ for reasons which will be explained at the very beginning of Section \ref{cl}.\par It follows from the tables that the index set $I$ has cardinality $r=0$, $1$, $2$, or $3$ (the latter case happens only for $\g=so(8)$). The case $r=0$, i.e.\ $\g^\natural=\{0\}$, happens if and only if $\g=spo(2|1)$. In the case when the center is non-zero (resp. zero) we use $I=\{0,1,\ldots,r-1\}$ (resp. $I=\{1,\ldots,r\}$) as the index set, and denote the center of $\g^\natural$ by $\g^\natural_0$.

Let $C_{\g^\natural_i}$  be the Casimir operator of $\g^\natural_i$ corresponding to $(\cdot|\cdot)_{|\g^\natural_i\times \g^\natural_i}$. We define the dual Coxeter number $h^\vee_{0,i}$ of $\g_i^\natural$ as half of the eigenvalue of $C_{\g^\natural_i}$  acting on $\g^\natural_i$ (which is $0$ if $\g_i^\natural$ is abelian). Their values are given in Table 4 of \cite{KW}.
If all $h_{0,i}, i\in I$, are equal, we denote this number by $h^\vee_0$.
\vskip10pt
In \cite{KWR},  \cite{KW}  a vertex algebra $W^{k} (\g, f)$, called a {\it universal W-algebra}, has been associated to each triple $(\g, f, k)$, where $\g$ is a basic simple Lie superalgebra  with a  non-degenerate invariant supersymmetric bilinear form, $f$ is a nilpotent element of $\g_{\bar 0}$, and $k\in \C$, by applying the quantum Hamiltonian reduction functor $H_f$ to the affine vertex algebra $V^k(\g)$ (see below for the definition of $V^k(\g)$ and \cite[Section 2]{KWR} for the definition of $H_f$). In particular, it was shown that, for $k$ non-critical, i.e., $k\ne - h^\vee$,  ${W}^{k}(\g,f)$ has  a Virasoro vector $\omega$, \cite[(2.2)]{KW}, making it  a conformal vertex algebra, and a set of free generators was constructed. For $k$ non-critical the vertex algebra 
$W^{k} (\g, f)$  has a unique simple quotient, denoted by $W_{k} (\g, f)$.\par
In greater detail,  paper \cite{KW}  studies the universal minimal W-algebras of level $k$, which correspond to $f=e_{-\theta}$ where $-\theta$ is a minimal root.  These W-algebras 
and their simple quotients will be of primary interest for the present paper.  To simplify notation, we set 
\begin{align*}
{W}^{k}(\g, \theta)=W^{k}(\g, e_{-\theta}),\quad
{W}_{k}(\g, \theta)=W_{k}(\g, e_{-\theta}).
\end{align*}

Throughout  the paper we shall assume that $(\theta|\theta)=2$ and that $k\ne - h^\vee$. Then the Virasoro vector $\omega$ of 
${W}^{k}(\g, \theta)$  has central charge \cite[(5.7)]{KW}
\begin{equation}\label{cgk}
c(\g,k)=\frac{k\,\sdim\g}{k+h^\vee}-6k+h^\vee-4.
\end{equation}\par
Thereafter,  we  adopt the following notation for the vertex operator corresponding to the state $a$:
$$
Y(a,z)=\sum_{n\in\ganz} a_{(n)}z^{-n-1}.
$$
 We frequently use the notation of the $\l$-bracket and of the  normally ordered product:
$$[a_\l b]=\sum_{n\geq 0}Ê\frac{\l^n}{n!} (a_{(n)} b),\qquad : ab: = a_{(-1)} b.$$ 
A defining set of axioms, satisfied by these operations, is described in \cite{BK2}. We frequently use these axioms in our calculations.\par
If the vertex algebra admits a Virasoro vector and $\D_a$ is the conformal weight of a state $a$, then we also write the corresponding vertex operator as 
$$
Y(a,z)=\sum_{m\in\ganz-\D_a} a_{m}z^{-m-\D_a},
$$ 
so that 
$$
a_{(n)} =a_{n-\D_a+1},\  n\in\ganz,\quad  a_m =a_{(m+\D_a-1)},\ m\in\ganz-\D_a.
$$
The vacuum vector will be denoted by $\vac$.
\vskip10pt

The following result is Theorem 5.1 of \cite{KW}  and gives the structure of    $W^{k}(\g, \theta)$. 
Notice that formula \eqref{GG} below is taken from \cite{KW2}. Notation is as in the Introduction.
\begin{theorem}\label{kac} (a) The vertex algebra ${W}^{k}(\g, \theta)$  is strongly and freely generated by elements  $J^{\{a\}}$, where $a$ runs over a basis of $\g^\natural$, $G^{\{v\}}$, where $v$ runs over a basis of $\g_{-1/2}$, and the
Virasoro vector  $\omega$.\par
(b) The elements $J^{\{a\}},\, G^{\{v\}}$ are primary of conformal weight $1$ and $3/2$,\, respectively, with respect to $\omega$.\par
(c) The following $\l$-brackets hold:
\begin{equation}\label{JJ}
[{J^{\{a\}}}_{\lambda}J^{\{b\}}]=J^{\{[a,b]\}}+\lambda\left((k+h^\vee/2)(a|b)-\tfrac{1}{4}\kappa_0(a,b)\right),\ a,b\in \g^\natural,
\end{equation}
where $\kappa_0$ is the Killing form of $\g_0$, and
\begin{equation}\label{JG}
[{J^{\{a\}}}_{\lambda}G^{\{u\}}]=G^{\{[a,u]\}},\ a\in \g^\natural,\,u\in\g_{-1/2}.
\end{equation}
(d) The following $\l$-brackets hold for $u, v\in\g_{-1/2}$:
\begin{align}\label{GG}
&[{G^{\{u\}}}_{\lambda}G^{\{v\}}]=-2(k+h^\vee)(e_\theta|[u,v])\omega+(e_\theta|[u,v])\sum_{\alpha=1}^{\dim \g^\natural} 
:J^{\{u^\alpha\}}J^{\{u_\alpha\}}:\\\notag
&+\sum_{\gamma=1}^{\dim\g_{1/2}}:J^{\{[u,u^{\gamma}]^\natural\}}J^{\{[u_\gamma,v]^\natural\}}:
+\lambda\sum_{\gamma=1}^{\dim\g_{1/2}} J^{\{[[u,u^\gamma],[u_\gamma,v]]^\natural\}}\\\notag
&+2(k+1)(\partial+2\lambda)J^{\{[[e_\theta,u],v]^\natural\}}+ \tfrac{\lambda^2}{3}
v \sum_{\gamma=1}^{\dim\g_{1/2}} c'([u,u^\gamma]^\natural,[u_\gamma,v]^\natural).\\\notag&+\tfrac{\lambda^2}{3}(e_\theta|[u,v])\left(-(k+h^\vee)c(\g,k)+(k+h^\vee/2) \sdim\g^\natural-\tfrac{1}{2}\sum_{i\in I}h_{0,i}^\vee \sdim\g_i^\natural\right).
\end{align}
Here $\{u_\alpha\}$ (resp. $\{u_\gamma\}$) is a basis of  $\g^\natural$ (resp. $\g_{1/2}$) and $\{u^\alpha\}$ (resp. $\{u^\gamma\}$) is the corresponding dual basis w.r.t. $(\cdot|\cdot)$ (resp w.r.t. $\langle\cdot,\cdot\rangle_{\rm ne}=(e_{-\theta}|[\cdot,\cdot])$),
$a^\natural$ is the orthogonal projection of $a\in\g_0$ to $\g^\natural$, and
$$c'(a,b)=
k(a|b)+\tfrac{1}{2}{\rm str}_{|\bigoplus\limits_{i>0}\g_{i}}(ad(a) ad(b)).
$$
\end{theorem}
\vskip10pt
Let $\aa$ be a Lie superalgebra equipped with a nondegenerate  invariant supersymmetric bilinear form $B$. The universal affine vertex algebra $V^B(\aa)$ is  the universal enveloping vertex algebra of  the  Lie conformal superalgebra $R=(\C[T]\otimes\aa)\oplus\C$ with $\lambda$-bracket given by
$$
[a_\lambda b]=[a,b]+\lambda B(a,b),\ a,b\in\aa.
$$
In the following, we shall say that a vertex algebra $V$ is an affine vertex algebra if it is a quotient of some $V^B(\aa)$.
If $\aa$ is equipped with a fixed bilinear form $(\cdot | \cdot)$ and $B=m (\cdot | \cdot)$, then we denote  $V^B(\aa)$ by  $V^m(\aa)$. Set $ \aa'=\aa\otimes\C[t,t^{-1}]\oplus \C K$ be the affinization of $\aa$. This is the Lie superalgebra with $K$ a central element and Lie bracket defined, for $a,a'\in\aa$,
$$
[a\otimes t^k,a'\otimes t^h]=[a,b]\otimes t^{h+k}+\d_{h,-k}(a|b)K.
$$
Set $\aa^+=span(a\otimes t^k\mid a\in\aa, \ k\ge0)$. Let $\C_m$ be the one-dimensional representation of $\aa^+\oplus\C K$ such that $a\otimes t^k\cdot1=0$ for all $a\in\aa$ and $k\ge0$ and $K\cdot1=m$. Recall that $V^m(\aa)$ can be realized as $U(\aa')\otimes_{U(\aa^++\C K)} \C_m$.  Regard $d=t\frac{d}{dt}$  as a derivation of $\aa'$. Then $-d$ induces a Hamiltonian operator $H$  whose eigenspace decomposition defines a $\ganz$-grading of $V^m(\aa)$. Since the eigenvalues of $H$ are positive and the zero eigenspace is $\C_m$, there is a unique proper  maximal graded ideal in $V^m(\aa)$. We let $V_m(\aa)$ denote the corresponding graded simple quotient. 

It is known that  $V^m(\aa)$ admits a unique graded irreducible quotient,  which we denote by $V_m(\aa)$.


Let $\mathcal{V}^k(\g^\natural)$ be the  subalgebra of the vertex algebra ${W}^{k}(\g, \theta)$,  generated by $\{J^{\{a\}}\mid a\in\g^\natural\}$.  The vertex algebra 
$\mathcal{V}^k(\g^\natural)$ is isomorphic to a universal affine vertex algebra. More precisely, letting 
\begin{equation}\label{ki}
k_i=k+\half({h^\vee-h_{0,i}^\vee}),\ i\in I,
\end{equation}
 the map $a\mapsto J^{\{a\}}$ extends, by \eqref{JJ},  to an isomorphism
$\mathcal{V}^k(\g^\natural)\simeq V^{B}(\g^\natural)$ where $B(a,b)=\d_{ij}k_i(a|b),\,a\in \g^\natural_i, b\in \g^\natural_j$. 
Note that

\begin{equation}\label{current}
\mathcal{V}^k(\g^\natural)\simeq \bigotimes_{i\in I}V^{k_i}(\g_i^\natural).
\end{equation}
We also set $\mathcal V_k(\g^\natural)$ to be the image of $\mathcal{V}^k(\g^\natural)$ in ${W}_{k}(\g, \theta)$. Clearly we can write
$$
\mathcal V_k(\g^\natural)\simeq \bigotimes_{i\in I} \mathcal V_{k_i}(\g_i^\natural),
$$
where $\mathcal V_{k_i}(\g_i^\natural)$ is some quotient (not necessarily simple) of $V^{k_i}(\g^\natural_i)$.
\vskip10pt
Recall that the fusion product  of two subsets $A,B$ of a vertex algebra $V$ is 
$$
A\cdot B=span(a_{(n)}b\mid n\in\ganz,\ a\in A,\ b\in B).
$$
This product is associative due to the Borcherds identity (cf. \cite{BK2}).\vskip5pt
Recall that a vector $v$ in a representation  of a vertex algebra $V^B(\aa)$ is {\it singular} if $a_{n}v=0$ for any $a\in\aa$ and $n>0$. We introduce a similar notion for quantum
affine $W$-algebras.
\begin{defi}
We will say that a vector $v\in W^k(\g, \theta)$ is singular if $a_{n} v=0$ for $n>0$ and $a$ in the set of generators $J^{\{b\}} (b\in\g^\natural), G^{\{u\}}\ (u\in\g_{-1/2}), \omega$.
\end{defi}
Note that if $v\notin \C \vac$ is singular, then $W^k(\g, \theta)\cdot v$ is a proper ideal of $W^k(\g, \theta)$. (It is a left ideal due to the associativity of the fusion product, and is also a right ideal since  
$W^k(\g, \theta)$ admits a Virasoro vector).
\section{Collapsing levels}\label{cl}
\begin{defi} If ${W}_{k}(\g, \theta)=\mathcal{V}_{k}(\g^\natural)$, we say that $k$  a {\sl collapsing level}.
\end{defi}
In this Section we  give a complete classification of collapsing  levels. Our main tool will be a certain degree two polynomial $p(k)$, which turns out to govern the 
first and the second product  of  the elements $G^{\{u\}},\, u\in\g_{-1/2}$, of $W^k(\g,\theta)$.\par Remark that if 
 $\g=sl(2)$, then $\g_{-1/2}=\g_{1/2}=\{0\}$, and
 ${W}^{k}(\g, \theta)$ is the universal Virasoro vertex algebra with central charge $\frac{3k}{k+2}-6k-2$, see \eqref{cgk}. Hence $k$ is a collapsing level in this case if and only if $k=-\half$ or $-\frac{4}{3}$, in which case $W_k(\g, \theta)=\C\vac$; also, only these $k$ are conformal. We shall therefore exclude this case from further considerations.
\par
We introduce the aforementioned polynomial $p(k)$ through the following structural lemma. For $a\in\g$, we denote by $P(a)= \overline 0$ or $\overline 1$ the parity of $a$.
\begin{lemma}\label{formG}
Let $u, v\in\g_{-1/2}$.  There is a monic polynomial $p(k)$ of degree two, independent of $u$ and $v$, such that
\begin{equation}\label{formG1}
{G^{\{u\}}}_{(2)}{G^{\{v\}}}=4(e_\theta|[u,v])p(k)\vac.
\end{equation}\par
Moreover, the linear polynomial $k_i$, $i\in I$, defined by \eqref{ki}, divides $p(k)$ and
\begin{equation}\label{result}
{G^{\{u\}}}_{(1)}{G^{\{v\}}}=4\sum_{i\in I} \frac{p(k)}{k_i}J^{\{([[e_\theta,u],v]])^\natural_i\}}
\end{equation}
where  $(a)^\natural_i$ denotes the orthogonal projection of $a\in \g_0$ onto $\g^\natural_i$.
\end{lemma}
\begin{proof}Recall formal \eqref{GG}.
If $\g^\natural$ is simple then (1) follows directly from the Remark in \cite{KW2}.
In the other cases,  let $f_k(u,v)$  be the constant such that ${G^{\{u\}}}_{(2)}{G^{\{v\}}}=f_k(u,v)\vac$. We observe that, by skew-symmetry of $\lambda$-product, $f_k(\cdot,\cdot)$ is an anti-symmetric form and, by Jacobi identity, $f_k(\cdot,\cdot)$ is $\g^\natural$-invariant. Since $\g_{1/2}$, as a $\g^\natural$--module, is either irreducible or a sum $U\oplus U^*$ with $U$ irreducible and $U$ inequivalent to $U^*$, we see that, up to a constant, there is a unique anti-symmetric $\g^\natural$-invariant nondegenerate bilinear form on $\g_{-1/2}$. Since $(e_\theta|[u,v])$ is such a form, we have that
$f_k(\cdot,\cdot)=4p(k)(e_\theta|[\cdot,\cdot])$. The fact that $p(k)$ is a monic polynomial of degree two follows immediately from \eqref{cgk} and \eqref{GG}.

We now prove \eqref{result}. If $\g^\natural=\{0\}$, i.e. if $\g=spo(2|1)$,  the left hand side of \eqref{result} is zero by \eqref{GG}, whereas the right hand side is obviously zero. Hence we can assume $\g^\natural\ne\{0\}$. Let $\{a^j_i\}$ be a basis of $\g^\natural_i$ and $\{b^j_i\}$ its dual basis with respect to $(\cdot|\cdot)_{|\g^\natural_i\times \g^\natural_i}$.
From \eqref{GG}, we see that
\begin{equation}\label{1productexp}
{G^{\{u\}}}_{(1)}{G^{\{v\}}}=\sum_{i,j}\left(4([[e_\theta,u],v]|b^j_i)k+c_{i,j}(u,v)\right)J^{\{a^j_i\}}
\end{equation}
with $c_{i,j}(u,v)$ a constant independent of $k$. By \eqref{JG},
${J^{\{b^j_i\}}}_{(1)}({G^{\{u\}}}_{(1)}{G^{\{v\}}})={G^{\{[b^j_i,u]\}}}_{(2)}{G^{\{v\}}}=4(e_\theta|[[b^j_i,u],v])p(k)\vac$. 

On the other hand, by \eqref{1productexp}, 
\begin{align*}
{J^{\{b^j_i\}}}_{(1)}({G^{\{u\}}}_{(1)}{G^{\{v\}}})&=\sum_{r,s}\left(4([[e_\theta,u],v]|b^r_s)k+c_{r,s}(u,v)\right){J^{\{b^j_i\}}}_{(1)}{J^{\{a_s^r\}}}\\
&=\left(4([[e_\theta,u],v]|b^j_i)k+c_{i,j}(u,v)\right)(-1)^{P(b_i^j)}k_i\vac,
\end{align*}

Hence
\begin{equation}\label{from1to2}
4(e_\theta|[[b^j_i,u],v])p(k)\vac=\left(4([[e_\theta,u],v]|b^j_i)k+c_{i,j}(u,v)\right)(-1)^{P(b^j_i)}k_i\vac.
\end{equation}
It is easily checked browsing Tables 1--3 that no component of $\g^\natural$ acts trivially on $\g_{-1/2}$, so for each $i\in I$ there is $j$ such that $[b^j_i,u]\ne0$. Since the form $(e_\theta|[\cdot,\cdot])$ is nondegenerate, we can choose $v$ so that $(e_\theta|[[b^j_i,u],v])\ne0$. It follows then at once from \eqref{from1to2} that $k_i$ divides $p(k)$.

 Write explicitly $p(k)=(k+a)k_i$. It follows from \eqref{from1to2} that, for all $i,j$ and all $u,v\in\g_{-1/2}$,
 $$
\left(4([[e_\theta,u],v]|b^j_i)k+c_{i,j}(u,v)\right)(-1)^{P(b^j_i)}=4(e_\theta|[[b^j_i,u],v])(k+a).
$$
It follows that $(e_\theta|[[b^j_i,u],v])=([[e_\theta,u],v]|b^j_i)(-1)^{b^j_i}$ and that 
$$c_{i,j}(u,v)=4(e_\theta|[[b^j_i,u],v])(-1)^{P(b^j_i)}a=4([[e_\theta,u],v]|b^j_i).$$
 Substituting in \eqref{1productexp} we get \eqref{result}.
\end{proof}

The polynomial $p(k)$ can be easily calculated for all $\g$. If $\g^\natural=\{0\}$ then, by \eqref{GG}, $p(k)=-\frac{1}{6}(k+h^\vee)c(\g,k)$. If $\g^\natural\ne\{0\}$ and the set $\{k_i\mid i\in I\}$  has only one element, then we can use formula 2 in  \cite[Remark]{KW2}. If $\{k_i\mid i\in I\}$ has more than one element then, by Lemma \ref{formG}, it has only two elements and $p(k)$ is their product. The outcome is listed in the following table.
\vskip10pt
\centerline{Table 4}
{\tiny
\vskip10pt
\centerline{\begin{tabular}{c|c||c|c}
$\g$&$p(k)$&$\g$&$p(k)$\\
\hline
$sl(m|n)$, $n\ne m$&$(k+1) (k+(m-n)/2)$&$E_6$&$(k+3) (k+4)$\\
\hline
$psl(m|m)$&$ k (k+1)$&$E_7$&$(k+4)(k+6)$\\
\hline
$osp(m|n)$&$(k+2) (k+(m-n-4)/2)$&$E_8$&$(k+6) (k+10)$\\
\hline
$spo(n|m)$&$(k+1/2) (k+(n-m+4)/4)$&$F_4$&$(k+5/2) (k+3)$\\
\hline
$D(2,1;a)$&$(k-a)(k+1+a)$&$G_2$&$ (k+4/3) (k+5/3)$\\
\hline
$F(4)$, $\g^\natural=so(7)$ & $(k+2/3)(k-2/3)$ &$G(3)$, $\g^\natural=G_2$ & $(k-1/2)(k+3/4)$  \\
\hline
$F(4)$, $\g^\natural=D(2,1;2)$ & $(k+3/2)(k+1)$ &$G(3)$, $\g^\natural=osp(3|2)$ & $(k+2/3)(k+4/3)$  \\
\end{tabular}}}
\vskip 15pt



In Proposition \ref{deligne} we will give a uniform formula for $p(k)$ for a certain class of pairs $(\g,\theta)$.

\begin{prop}\label{g2g}
The following are equivalent:
\begin{enumerate}
\item There is $v\in \g_{-1/2}$, $v\ne 0$, such that $G^{\{v\}}=0$ in ${W}_{k}(\g, \theta)$.
\item There is $v\in \g_{-1/2}$, $v\ne0$, such that ${G^{\{u\}}}_{(2)}{G^{\{v\}}}=0$ in ${W}^{k}(\g, \theta)$ for all $u\in \g_{-1/2}$.
\item\label{u(1)v=0} ${G^{\{v\}}}_{(2)}{G^{\{u\}}}=0$ in ${W}^{k}(\g, \theta)$ for all $v,u\in \g_{-1/2}$.
\item $G^{\{v\}}=0$ in ${W}_{k}(\g, \theta)$ for all $v\in\g_{-1/2}$.
\end{enumerate}
In particular, $G^{\{v\}}=0$ in  ${W}_{k}(\g, \theta)$ for all $v\in\g_{-1/2}$ if and only if  $p(k)=0$.
\end{prop}
\begin{proof}
(1) $\Rightarrow$ (2): if $G^{\{v\}}=0$ in ${W}_{k}(\g, \theta)$ for some $v\in\g_{-1/2}$, then, by Lemma \ref{formG}, $(e_\theta|[u,v])p(k)=0$ for all $u\in\g_{-1/2}$, thus ${G^{\{u\}}}_{(2)}{G^{\{v\}}}=0$ in ${W}^{k}(\g, \theta)$ for all $u\in \g_{-1/2}$. 

(2) $\Rightarrow$ (3):
since $(e_\theta|[\cdot,\cdot])$ is nondegenerate, if (2) holds, then $p(k)=0$ so ${G^{\{v\}}}_{(2)}{G^{\{u\}}}=0$ in ${W}^{k}(\g, \theta)$ for all $v,u\in \g_{-1/2}$.

(3) $\Rightarrow$ (4): assume first ${G^{\{v\}}}_{(i)}G^{\{u\}}=0$ for all $v\in\g_{-1/2}$ and $i> 0$. Then, since ${J^{\{a\}}}_{(i)}G^{\{u\}}=0$ for all $a\in\g^\natural$ and $i>0$, we have that $G^{\{u\}}$ is a singular vector in ${W}^{k}(\g, \theta)$, so $G^{\{u\}}$ is in the maximal proper ideal of ${W}^{k}(\g, \theta)$. 

If, otherwise, there is   $v$ in $\g_{-1/2}$ such that ${G^{\{v\}}}_{(i)}G^{\{u\}}\ne 0$ for some $i>0$, then,  since we are assuming \eqref{u(1)v=0},  we have ${G^{\{v\}}}_{(1)}G^{\{u\}}\ne 0$. By \eqref{GGsimplified} we know that ${G^{\{v\}}}_{(1)}G^{\{u\}}=\sum_i J^{\{a_i\}}$ with $a_i\in\g^\natural_i$. In particular $\g^\natural\ne \{0\}$. Assume $a_j\ne0$. 
Observe that ${J^{\{a\}}}_{(1)}({G^{\{v\}}}_{(1)}G^{\{u\}})={G^{\{[a,v]\}}}_{(2)}G^{\{u\}}=0$ for all $a\in\g^\natural$. Choosing $b_j\in\g^\natural_j$ such that $(b_j|a_j)=1$ and arguing as in the proof of \eqref{result} in Lemma \ref{formG} we see that
$$0={J^{\{b_j\}}}_{(1)}({G^{\{v\}}}_{(1)}G^{\{u\}})={J^{\{b_j\}}}_{(1)} \left(\sum_i J^{\{a_i\}}\right)=k_j\vac,
$$
hence ${J^{\{a\}}}_{(1)}J^{\{a_j\}}=0$ for all $a\in\g^\natural$.  
Since, by skew-symmetry, 
$${G^{\{w\}}}_{(i)}J^{\{a_j\}}=-\delta_{i,0}{J^{\{a_j\}}}_{(0)}G^{\{w\}},\quad(i\ge 0),$$ we see that $J^{\{a_j\}}$ is a singular vector in ${W}^{k}(\g, \theta)$, thus $J^{\{a_j\}}$ belongs to the maximal proper ideal of ${W}^{k}(\g, \theta)$. Since $\g^\natural_j$ is simple if $j>0$ and $\g^\natural_0$ has at most dimension one, we see that, if $a_j\ne 0$ then $J^{\{a\}}$ is in the maximal proper ideal of ${W}^{k}(\g, \theta)$ for all $a\in\g^\natural_j$. 

Assume that  $u$ is a highest weight weight vector of weight $\a$ for an irreducible component of $\g_{-1/2}$. 
 By looking at Tables 1--3, we see that  $\g^\natural$ does not act trivially on $\g_{1/2}$, hence there is $h\in\h^\natural$ such that $[h,u]=\a(h)u$ with $\a(h)\ne 0$. Since ${J^{\{h\}}}_{(0)} G^{\{u\}}=\a(h) G^{\{u\}}$ is   in the maximal proper ideal of ${W}^{k}(\g, \theta)$, it follows that $G^{\{u\}}$ is zero in ${W}_{k}(\g, \theta)$. By the same argument it is clear that, if the highest weight vector of a component of $\g_{-1/2}$ is in a proper ideal, then the whole component is, thus $G^{\{u\}}=0$ in ${W}_{k}(\g, \theta)$ for all $u\in\g_{-1/2}$.

(4) $\Rightarrow$ (1): obvious.
 \end{proof}


\begin{theorem}\label{TT}  ${W}_{k}(\g, \theta)=\mathcal{V}_{k}(\g^\natural)$,  i.e., $k$  is a collapsing level,  if and only if 
 $p(k)=0$. In such cases, \begin{equation}\label{collapse}
{W}_{k}(\g, \theta)=\bigotimes_{i\in I: k_i\ne0}V_{k_i}(\g^\natural_i).
\end{equation}
\end{theorem}
\begin{rem}
Notice that for some $i$, it may hold $k_i+h^\vee_i=0$. Since we are assuming $k+h^\vee\ne0$, we still have the Virasoro vector $\omega$ acting on $\mathcal V^k(\g^\natural)$. Since $a\in\g^\natural$ has conformal weight $1$, $\omega_0$ acts on $\mathcal V^k(\g^\natural)$ as the Hamiltonian operator $H$. 
\end{rem}
\begin{proof}[Proof of Theorem \ref{TT}] Assume ${W}_{k}(\g, \theta)=\mathcal{V}_{k}(\g^\natural)$. Then, for $v\in \g_{-1/2}$, $G^{\{v\}}$ belongs to $\mathcal{V}_{k}(\g^\natural)$. Since conformal weights in  $\mathcal{V}_{k}(\g^\natural)$ are integral  and the conformal weight of $G^{\{v\}}$ is $3/2$,  we have $G^{\{v\}}=0$ and,  by Proposition  \ref{g2g}, $p(k)=0$. Conversely, if $p(k)=0$, by 
 Proposition \ref{g2g} all $G^{\{v\}}$ are $0$.  By \eqref{GGsimplified}, we have that $\omega\in \mathcal V_k(\g^\natural)$, hence  ${W}_{k}(\g, \theta)=\mathcal{V}_{k}(\g^\natural)$. \par
 We now prove \eqref{collapse}. Let $\mathcal I$ be the maximal proper ideal of $W^k(\g,\theta)$. Since $\omega_0$ acts as $H$ on  $\mathcal V^k(\g^\natural)$, we have that $\mathcal I\cap\mathcal V^k(\g^\natural)$ is a graded proper ideal of $\mathcal V^k(\g^\natural)$. 
If  $p(k)=0$, since ${W}_{k}(\g, \theta)$ is simple, we have that $\mathcal V_k(\g^\natural)$ is the simple graded quotient of $\mathcal V^k(\g^\natural)$,  hence
$$\mathcal V_k(\g^\natural)\simeq \bigotimes_{i\in I}V_{k_i}(\g_i^\natural).
$$
If $k_i=0$, then $V_{k_i}(\g^\natural_i)=\C\vac$, so 
the right hand side is the same as in \eqref{collapse}.
\end{proof}

\begin{rem} \label{ArM}
We note that in the case  when $\g$ is   of type $D_n$, Theorem \ref{TT} implies that
$${W}_{-2}(\g) = V_{n-4} (sl(2)). $$
Since it is well-known that $V_{n-4} (sl(2))$ is a rational vertex algebra for $n \ge 5$, we slightly refine \cite[Theorem 1.2 (2)]{AM} in this case.
\end{rem}

In the following result we determine the levels $k$ for which ${W}_{k}(\g, \theta)=\C\vac$. The Lie algebra case has been already discussed in \cite[Theorem 7.2]{AM}.
\begin{prop}\label{ArakawaMoreau}
${W}_{k}(\g, \theta)=\C\vac$ if and only if 
\begin{enumerate}
\item \label{CDeligne}$k=-\frac{h^\vee}{6}-1$ and either $\g$ is  a Lie algebra  belonging to the Deligne's series (i.e. $\g$ is of type $A_2,G_2,D_4,F_4,E_6,E_7,E_8$, \cite{D}) or $\g=psl(m|m)$ ($m\ge2$), $\g=osp(n+8|8)$ ($n\ge 2$), $\g$ is of type $F(4)$, $G(3)$ (for both choices of $\theta$), or $\g=spo(2|1)$;
\item\label{Cspo}\label{Csp(2|1)} $k=-\frac{1}{2}$ and  $\g=spo(n|m)\ (n\geq 1)$.
\end{enumerate} 
\end{prop}
\begin{proof}By Theorem \ref{TT}, ${W}_{k}(\g, \theta)=\C\vac$ if and only if  $k$ is a collapsing level and $k_i=0$  for all $i$. If $\g=spo(2|1)$, then the latter condition is empty ($\g^\natural=\{0\}$) so, in this special case, ${W}_{k}(\g, \theta)=\C\vac$ if and only if   $p(k)=0$. This proves  that both $-\frac{h^\vee}{6}-1$ and $-\half$ are collapsing levels for $spo(2|1)$. 

If $\g\ne \{0\}$, then $k_i=0$  for all $i$ if and only if $h^\vee_{0,i}$ have the same value for all $i$ (denoted by $h^\vee_0$). A direct inspection, using Table 4 from \cite{KW},  shows that this happens only in the cases listed in \eqref{CDeligne}, \eqref{Cspo} above. Moreover,  we must have $k_i=k+\frac{h^\vee-h_{0}^\vee}{2}=0$, so 
\begin{equation}\label{ki=0}
k=-\frac{h^\vee-h_{0}^\vee}{2}. 
\end{equation} For the cases listed in \eqref{CDeligne}, we checked case by case that 
\begin{equation}\label{delignehzero}h_{0}^\vee=\frac{2}{3}h^\vee-2
\end{equation} while for $\g=spo(n|m)$, 
\begin{equation}\label{spohzero}
h_{0}^\vee=h^\vee-1.
\end{equation} Plugging in \eqref{ki=0}, we obtain the values of $k$ in (1) and (2) . \end{proof}
In the following Proposition we obtain certain uniform formulas for $\sdim \g$ and $p(k)$ (see \cite{D}, \cite{Wes}, \cite{Kawa}).
\begin{prop}\label{deligne} If $(\g,\theta)$ is as in Proposition \ref{ArakawaMoreau}  \eqref{CDeligne}, then 
\begin{equation}\label{sdimdeligne}\sdim\, \g=\frac{2 (h^\vee + 1) (5 h^\vee  - 6)}{h^\vee + 6},\quad p(k)=(k+\frac{h^\vee}{6}+1)(k+\frac{h^\vee}{3}).
\end{equation}
If  $\g=spo(n|m)$, then 
\begin{equation}\label{sdimspo}\sdim\,\g=(2h^\vee-1)(h^\vee-1),\quad p(k)=(k+\frac{1}{2})(k+\frac{h^\vee+1}{2}).
\end{equation}
\end{prop}
\begin{proof}
If ${W}_{k}(\g, \theta)=\C\vac$, then $c(\g,k)=0$. Substituting in \eqref{cgk} the values of $k$ given in Proposition \ref{ArakawaMoreau}  and then solving for $\sdim\,\g$ in the equation $c(\g,k)=0$, one obtains  the formula for $\sdim\, \g$ given in \eqref{sdimdeligne} for the Lie superalgebras $\g$ listed in Proposition \ref{ArakawaMoreau}  \eqref{CDeligne} and the formula for $\sdim\,\g$ given in \eqref{sdimspo} for $\g=spo(n|m)$.

We have  
$\kappa_0(a,b)=2h^\vee_0(a|b)$ for $a,b\in\g^\natural$, hence we can apply
 \cite[Remark]{KW2} and obtain that
\begin{align*}{G^{\{u\}}}&_{(2)}G^{\{v\}}=\\&-\frac{2}{3}(e_\theta|[u,v])\{(k+h^\vee)c(\g,k)-(k+\frac{h^\vee-h^\vee_0}{2})(\sdim\,\g_0+\sdim\,\g_{1/2})\}\notag.
\end{align*}
Hence, by \eqref{formG1},  
\begin{equation}\label{cuv}
p(k)=-\frac{1}{6}\{(k+h^\vee)c(\g,k)-(k+\frac{h^\vee-h^\vee_0}{2})(\sdim\,\g_0+\sdim\,\g_{1/2})\}.
\end{equation}
  Since $\sdim\,\g_0=\sdim\g-2\sdim\,\g_{1/2}-2$ and $\sdim\g_{1/2}=2h^\vee-4$ \cite{KW}, plugging into \eqref{cuv} either \eqref{delignehzero} or \eqref{spohzero} one gets also our formulas for $p(k)$.
  \end{proof}
\begin{rem} Note also that for $\g=sl(m|n)$ ($m\ne n$) and for $\g=osp(m|n)$ one has respectively 
\begin{align*} &\sdim\, \g={h^\vee}^2-1, &&p(k)= (k+1)(k+\frac{h^\vee}{2}),\\
&\sdim\, \g=\frac{(h^\vee+1)(h^\vee+2)}{2}, && p(k)=(k+2)(k+\frac{h^\vee-2}{2}).
\end{align*}
\end{rem}
\vskip10pt
 Let us conclude this Section with one application of our results on collapsing levels. Let $V_c$ denote the simple Virasoro vertex algebra of central charge $c$. Let $M(1)$ be the Heisenberg vertex algebra of central charge $1$. \par
Recall from Section \ref{uno} that  $H_{e_{-\theta}}$ denotes  the quantum reduction functor  from the category of $V^{k}(\g)$--modules to the category of ${W}^{k}(\g, \theta)$--modules, and that 
\begin{equation}\label{wkgt} H_{e_{-\theta}}( V^k (\g) ) =  {W}^{k}(\g, \theta). \end{equation}
It was proved in \cite{Araduke} that $H_{e_{-\theta}}$ is an  exact functor and it  maps an irreducible highest weight $V^k(\g)$--modules $M$ either to zero or to an irreducible highest weight $W^{k} (\g, \theta)$--module;  furthermore, one has that, if $M$ is not integrable, then $H_{e_{-\theta}}(M)$ is nonzero. This implies that, if $\g$ is a Lie algebra and $k \notin {\Z}_{\ge 0}$, then
 $$  H_{e_{-\theta}} ( V_k (\g) ) =  {W}_{k}(\g, \theta).  $$

 \begin{prop}\ 
 \begin{enumerate}
 \item \label{functor1}Let $\g =sl(2m)$. Then there is an exact functor from the category of $V_{-m}(\g)$--modules to the category of $V_{c=1}$--modules which maps  $V_{-m}(\g)$ to $V_{c=1}$.
\item\label{functor2} Let $\g = sp(2n)$. Then there is an exact functor from the category of $V_{-1 - \frac{n}{2}}(\g)$--modules to the category of $V_{c=-2}$--modules which maps  $V_{-1 - \frac{n}{2}}(\g)$ to $V_{c=-2}$.
\item \label{functor3} Let $\g = so(2n)$ with $n$ odd,  $n \ge 3$. Then there is an exact functor from the category of $V_{- (n-2)}(\g)$--modules to the category of $M(1) $--modules which maps  
$V_{- (n-2)}(\g)$ to $M(1)$.
\end{enumerate}
 \end{prop}
 \begin{proof}
\eqref{functor1} The case $m=1$ is clear. 
Assume that $m \ge 2$. First we notice that $k= -m$ is  a collapsing level and that $W_k(\g, \theta) = V_{-m +1} (sl(2m-2))$. Set $H_1=H_{e_{-\theta}}$. By \eqref{wkgt},  $H_1$ is an exact functor from  the category of $V_{-m} (sl(2m))$--modules to the category of $V_{-m +1} (sl(2m-2))$-modules.
By the  same argument we get an exact functor $H_2$  from  the category of $V_{-m+1} (sl(2m-2))$--modules to the category of $V_{-m +2} (sl(2m-4))$--modules. By repeating this procedure we get exact functors
$ H_i,\ i=1, \dots, m-1$ from  the category of $V_{-m-1+i} (sl(2m+2 -2i) )$--modules to the category of $V_{-m +i} (sl(2m-2i))$--modules.  Finally we use the  functor $H_{Vir}$ from the category of $V_{-1} (sl_2)$--modules to the category of $V_{c=1}$--modules.
Now the composition
$ H_{Vir} \circ H_{m-1} \circ \cdots \circ H_1$
gives the required exact functor.

\eqref{functor2} Since  $k$ is a collapsing level, we have $W_k(sp(2n), \theta) = V_{k+1/2} (sp(2n-2))$. But $k+1/2$ is a collapsing level for $sp(2n-2)$ and we finish  the proof as in  
case \eqref{functor1}.

\eqref{functor3} The proof  is similar to \eqref{functor1} and \eqref{functor2} and uses the fact that $W_{-1} (so(6), \theta)$ $=W_{-1} (sl(4), \theta) = M(1)$.
 \end{proof}
 Part \eqref{functor1} has been obtained by different methods in Arakawa-Moreau \cite{AM2}.
 \begin{rem}
  There is a different method for proving that a certain   level is collapsing. One can take a  non-trivial ideal $I$ in the  universal affine vertex algebra $V^k (\g)$ (usually generated by singular vector of conformal weight $2$), apply the BRST functor  $H_{e_{-\theta}}$,  and show that all  generators $G^{\{ u \} }$ of conformal weight $3/2$  of  $W^k (\g, \theta)$ lie in the non-trivial ideal  $H_{e_{-\theta}} (I)$ of 
  $W^k (\g, \theta)$. This will imply that $G^{\{ u \} } = 0$  in $ W_k (\g, \theta)$. Such approach was applied  in \cite{AM2} for $W$-algebras of type $A$, but can be also applied for all collapsing levels classified in our paper.  In particular, our results show that for every collapsing level there exists a non-trivial ideal $I$ in $V^k(\g)$ such that generators $G^{\{ u \} }$ lie in $H_{e_{-\theta}} (I)$.
  
   \end{rem}

\section{Conformal embeddings  of affine vertex algebras in W-algebras}\label{section4}
In this Section we introduce  a notion of a conformal embedding for $\mathcal V_{k}(\g^\natural)$ into ${W}_{k}(\g, \theta)$. Our first task is to endow $\mathcal V_{k}(\g^\natural)$ with a   Virasoro vector. In most cases this Virasoro   vector is the image in ${W}_{k}(\g, \theta)$ of the  Virasoro vector for  $\mathcal V^{k}(\g^\natural)$ given by the Sugawara construction.
On the other hand it might happen that $\mathcal V_{k}(\g^\natural)$ still has a Virasoro vector though Sugawara construction for  $\mathcal V^{k}(\g^\natural)$ does not make sense. This motivates the following definitions.\par
If $k_i+h^\vee_{0,i}\ne 0$,  then $\mathcal{V}^k(\g^\natural_i)$ is equipped with a Virasoro vector \begin{equation}\label{sug}
\omega_{sug}^i=\frac{1}{2(k_i+h_{0,i}^\vee)}\sum_{j=1}^{\dim \g^\natural_i} :J^{\{b^j_i\}}J^{\{ a_i^j\}} :,
\end{equation}
where $\{a^j_i\}$ is a basis of $\g^\natural_i$ and $\{b^j_i\}$ is its dual basis with respect to $(\cdot|\cdot)_{|\g^\natural_i\times \g^\natural_i}$.
If $k_i+h^\vee_{0,i}\ne 0$for all $i$,  we set
$$\omega_{sug}=\sum_{i\in I}\omega_{sug}^i.$$
Define
$$\mathcal K=\{k\in\C\mid k+ h^\vee\ne 0, k_i+h^\vee_{0,i}\ne 0 \text{ whenever $k_i\ne 0$}\}.$$
If $k\in \mathcal K$ we also set
$$\omega'_{sug}=\sum\limits_{i\in I:k_i\ne 0}\omega_{sug}^i.$$ 
Clearly, if $k_i\ne0$ for all $i\in I$ then $\omega_{sug}=\omega'_{sug}$.
Note that $\omega_{sug}=\omega'_{sug}$ in ${W}_{k}(\g, \theta)$ also when  $k_i+h^\vee_{0,i}\ne 0$ for all $i\in I$. Indeed, if there is $j \in I$ such that $k_{j}=0$, by Theorem \ref{TT}, $\mathcal V_{k_{j}}(\g^\natural_{j})$ collapses to $\C\vac$ so $\omega^{j}_{sug}$ collapses to zero. 
If instead $k_i=k_i+h^\vee_{0,i}=0$ then, by Theorem \ref{TT},   $\omega'_{sug}$ still defines a Virasoro vector for $\mathcal V_k(\g^\natural)$, but $\omega_{sug}$ is not well defined.

Finally, we define
$$
c_{sug}=
\text{ central charge of $\omega'_{sug}$.}$$

\begin{defi} Assume $k\in \mathcal K$. We say that $\mathcal V_k(\g^\natural)$ is conformally embedded in ${W}_{k}(\g, \theta)$ if $\omega'_{sug}=\omega$.
The level $k$ is called a {\sl conformal level}.
\end{defi}
The goal of this Section is to determine the conformal levels. An easy result in this direction is the following Proposition which explains why we replace $\omega_{sug}$  by   $\omega'_{sug}$.
\begin{prop} Assume that $k$ is a collapsing level in $\mathcal K$. Then $k$ is conformal.
\end{prop}
\begin{proof} If $k$ is a collapsing level, then $\omega\in \mathcal{V}_{k}(\g^\natural)$. Since $J^{\{a\}}, a\in \g^\natural$ is primary of conformal weight $1$ for $\omega$, we see that $[(\omega-\omega'_{sug})_{\lambda}J^{\{a\}}]=0$ for all $a\in\g^\natural$. This implies that ${J^{\{a\}}}_{(n)}(\omega-\omega'_{sug})=0$ for $n>0$, so  $\omega-\omega'_{sug}$ is a singular vector. Since $\omega-\omega'_{sug}$ has conformal weight $2$ for $\omega$, we see that $\omega-\omega'_{sug}\ne\vac$, hence, since  $\mathcal{V}_{k}(\g^\natural)$ is simple, $\omega-\omega'_{sug}=0$. 
\end{proof}

\vskip5pt
We now start  looking for conformal non--collapsing levels. 
Set 
\begin{equation}\label{J} \mathcal{J}^{k} (\g) = {W}^{k}(\g, \theta)\cdot(\omega- \omega_{sug}),\quad  \mathcal {W}^{k}(\g, \theta)= {W}^{k}(\g, \theta)/ \mathcal{J}^{k} (\g).
\end{equation}

  We have:
\begin{prop}\label{T} Assume $k_i+h^\vee_{0,i}\ne0$ for all $i\in  I $. The following conditions are equivalent:
\begin{enumerate} 
\item \label{conf-level}  $(\omega_{sug})_0 G^{\{u\}}=\tfrac{3}{2}G^{\{u\}}$ in ${W}^{k}(\g, \theta)$ for all $u\in \g_{-1/2}$.
\item \label{conff1-gen}The vector $ \omega- \omega_{sug}$
is a singular vector in ${W}^{k}(\g, \theta)$.
\item  \label{non-triv}
$G^{\{v\}}\ne 0$ in $\mathcal {W}^{k}(\g, \theta)$ for all $v\in\g_{-1/2}$, hence $ \mathcal{J}^{k} (\g)$
is a proper ideal in ${W}^{k}(\g, \theta)$.
\end{enumerate}
\end{prop}
\begin{proof}
 (\ref{conf-level}) $\implies$ (2).
Recall  that $J^{\{a\}}$ is primary of conformal weight $1$ for $\omega$ for all $a\in\g^\natural$. This, by \eqref{sug}, also implies that $\omega_{sug}$ is primary of conformal weight $2$ for $\omega$. By \eqref{current}, if $n\ge 0$,  $(\omega_{sug})_nJ^{\{a\}}=\delta_{n,0}J^{\{a\}}$. By the right Wick formula, $(\omega_{sug})_nG^{\{v\}}=0$ for $n>0$ and, by hypothesis, $(\omega_{sug})_0G^{\{v\}}=\omega_0G^{\{v\}}=\tfrac{3}{2}G^{\{v\}}$. It follows that $(\omega-\omega_{sug})_{(n)}X=0$ for $X=J^{\{a\}}$ or $X=G^{\{v\}}$ and $n>0$. Thus, by skewsymmetry of the $\lambda$-bracket, $X_{(n)}(\omega-\omega_{sug})=0$, for $n>0$ so 
$Z_{n}(\omega-\omega_{sug})=0$ for $n>0$ and $Z$ any generator. 

(2) $\implies$ (3). By (2), since $J^{\{b\}} (b\in\g^\natural), G^{\{u\}}\ (u\in\g_{-1/2}), \omega$ strongly generate ${W}^{k}(\g, \theta)$, we have that any element of ${W}^{k}(\g, \theta)\cdot(\omega-\omega_{sug})$ can be written as a sum of elements of type
$$
a^1_{s_1}\cdots a^t_{s_t}\cdot(\omega-\omega_{sug})
$$ with $a^i$ in the given set of generators and $s_i\le 0$ for all $i$.

 Since $a^1_{s_1}\cdots a^t_{s_t}\cdot(\omega-\omega_{sug})$ has conformal weight 
$$
-\sum_i s_i +2\ge 2,
$$
 we see that ${W}^{k}(\g, \theta)\cdot(\omega-\omega_{sug})$ cannot contain $G^{\{v\}}$.

(3) $\implies$ (1).  Recall from \eqref{sug} the definition of $\omega_{sug}^i$. From \eqref{JG}, we see that $\sum_{j=1}^{\dim \g^\natural_i} :J^{\{b^j_i\}}J^{\{ a_i^j\}} :_{0}$ acts on  $G^{\{v\}}$ as the Casimir operator of $\g^\natural_i$ acting on $v$. Since $\g_{-1/2}$ is either irreducible or is a direct sum of two contragredient $\g^\natural$-modules, we have  that $G^{\{v\}}$ is an eigenvector for $\sum_{j=1}^{\dim \g^\natural_i} :J^{\{b^j_i\}}J^{\{ a_i^j\}} :_{0}$     hence also for $(\omega_{sug})_0=\sum_i(\omega_{sug}^i)_0$.  Since $G^{\{v\}}\ne 0$  and $\omega=\omega_{sug}$ in ${W}^{k}(\g, \theta)$, the corresponding eigenvalue must be $\tfrac{3}{2}$.
\end{proof}

 The next result gives a classification of conformal embedding at non-collapsing levels. If $V(\mu)$, $\mu\in (\h^\natural)^*$, is an irreducible  $\g^\natural$--module,  then we can write
\begin{equation}\label{mui}
V(\mu)=\bigotimes_{j\in I}V_{\g^\natural_j}(\mu^{j}),
\end{equation}
 where $V_{\g^\natural_j}(\mu^{j})$ is an  irreducible $\g_j^\natural$--module. Let $\rho_0^j$ be the Weyl vector in $\g^\natural_j$ (with respect to the positive system induced by the choice of positive roots for $\g$).
 
 \begin{theorem} Let $k\in \mathcal K$ be a non--collapsing level.  Then   $k$ is a conformal level if and only if  $k_i\ne 0$ for all $i$ and 
 \begin{equation}\label{tremezzi}
\sum_{i\in I}\frac{(\mu^i|\mu^i+2\rho^i_0)}{2(k_i+h^\vee_{0,i})}=\frac{3}{2},
\end{equation}
where $V(\mu)$ is   an irreducible component of the $\g^\natural$-module $\g_{-1/2}$ and the $\mu^j$are defined in \eqref{mui}.
 \end{theorem}
 \begin{proof}First of all we note that, if $k\in\mathcal K$ is a non-collapsing level and $k_i\ne0$ for all $i$, then \eqref{tremezzi} is equivalent to condition \eqref{conf-level} of Proposition \ref{T}. Indeed, by \eqref{JG}, the map $u\mapsto G^{\{u\}}$ from $\g_{-1/2}$ to $span(G^{\{u\}}\mid u\in\g_{-1/2})$ is a map of $\g^\natural$-modules and, since the level is non--collapsing, it is an isomorphism. By \eqref{JG} again, ${J^{\{a\}}}_{(n)}G^{\{v\}}=0$ for all  $a\in\g^\natural$, $v\in\g_{-1/2}$, and $n>0$. It follows that $(\omega_{sug})_0$ acts on $G^{\{v\}}$ as $\sum_{i\in I}\frac{1}{k_i+h^\vee_{0,i}}C_{\g^\natural_i}$. Thus,  if $v$ belongs to  the irreducible component $V(\mu)$ of $\g_{-1/2}$, then the left hand side of \eqref{tremezzi}  gives the eigenvalue for the action of  $(\omega_{sug})_0$ on $G^{\{v\}}$. Since $\g_{-1/2}$ is either irreducible or the sum of two contragredient irreducible $\g^\natural$-modules, the left hand side of \eqref{tremezzi}  gives the eigenvalue of $(\omega_{sug})_0$ on $G^{\{v\}}$ for all $v\in\g_{-1/2}$.
 
 Assume that $k$ is a conformal non-collapsing level. By formula \eqref{result}, if $k_i=0$ for some $i$, then $k$ is a collapsing level, thus $k_i\ne 0$ for all $i$. 
Since the level is non--collapsing, we see from Theorem \ref{TT} that $p(k)\ne0$. It follows from Proposition \ref{g2g} that, for all $v\in \g_{-1/2}$, $G^{\{v\}}$ is nonzero in ${W}_{k}(\g, \theta)$.
Since $k$ is a conformal level, $\mathcal J^k(\g)$ is contained in the maximal proper ideal of  ${W}^{k}(\g, \theta)$. Hence  there is an onto map $\mathcal {W}^{k}(\g, \theta)\to {W}_{k}(\g, \theta)$, and in particular  $G^{\{v\}}$ is also nonzero in $\mathcal{W}^{k}(\g, \theta)$. From Proposition \ref{T}, we deduce that 
$(\omega_{sug})_0G^{\{v\}}=\tfrac{3}{2}G^{\{v\}}$ 
hence, for any irreducible component $V(\mu)$ of the $\g^\natural$-module $\g_{-1/2}$, we have that \eqref{tremezzi} holds.

Conversely, if $k_i\ne 0$ for all $i$ and \eqref{tremezzi} holds for a component $V(\mu)$ of $\g_{-1/2}$, then,
as already observed,  \eqref{tremezzi} holds for all components of $\g_{-1/2}$. Thus \eqref{conf-level} of  Proposition \ref{T} holds. It follows that $\mathcal{J}^{k} (\g)$ is a proper ideal, hence $\omega=\omega_{sug}$ in ${W}_{k}(\g, \theta)$.
\end{proof}

\begin{cor} \label{non-simple}
Assume that $k$ is a collapsing level such that relation (\ref{tremezzi}) holds. Then $\mathcal W^ k (\g)$ is not simple.
\end{cor}
\begin{proof}
Since $k$ is a collapsing level we have that $ G^{ \{ u \} } = 0 $ in ${W}_{k}(\g, \theta)$.  But Proposition \ref{T}  gives that $ G^{ \{ u \} } \ne  0 $ in ${W}^{k}(\g, \theta)$. The proof follows.
\end{proof}

Solving \eqref{tremezzi} case-by-case, one finds that, if $k$ is a solution, then   $k\in\{-\frac{2}{3}h^\vee,-\frac{h^\vee-1}{2}\}$. We do not have a uniform explanation for this striking coincidence, but we can provide a proof for broad classes of examples with a limited amount  of case-by-case checking. In the next Propositions we perform these calculations and list all conformal non--collapsing levels in all cases.


\begin{prop}\label{conformalweightsimple} Assume that $\g^\natural$ is either zero  or simple or 1-dimensional. 

If $\g=sl(3)$, or $\g=spo(n|n+2)$ with $n\ge 2$, $\g=spo(n|n-1)$ with $n\ge 2$, $\g=spo(n|n-4)$ with $n\ge 4$, then there are no non--collapsing conformal levels.

In all other cases the non-collapsing  conformal levels  are 
\begin{enumerate}
\item $k = -\frac{ h^{\vee} -1 }{2}$ if $\g$ is of type $G_2, F_4, E_6, E_7, E_8, F(4),G(3)$ (for both choices of $\theta$), or $\g=psl(m|m)$ ($m\ge2$);
\vskip 3pt
\item $k=-\frac{2}{3} h^{\vee}$ if  $\g=sp(n)\ (n\ge 6)$, or $\g=spo(2|m)\ (m\ge 2)$, or $\g=spo(n|m)\ (n\geq 4)$.
\end{enumerate}
\end{prop}
\begin{proof}  If $\g^\natural=\{0\}$ (i.e. $\g=spo(2|1)$) then, clearly, \eqref{tremezzi} is never satisfied.
 In all other cases $\g_{-1/2}$, as a $\g^\natural$--module,  is either irreducible or the sum of two  contragredient modules. 
 Let $V(\mu)$, $\mu\in (\h^\natural)^*$, be an irreducible component of $\g_{-1/2}$ as a $\g^\natural$--module. Let $\rho_0$ be the Weyl vector in $\g^\natural$ (with respect to the positive system induced by the choice of positive roots for $\g$). Then \eqref{tremezzi} reduces to 
 \begin{equation}\label{1}
\frac{(\mu|\mu+2\rho_0)}{2(k+\half(h^\vee+h_{0}^\vee))}=\frac{3}{2}.
\end{equation}

Let $C_{\g_0},C_{\g^\natural}$ be the Casimir elements of $\g_0$, $\g^\natural$ respectively. Clearly
$$
C_{\g_0}=\frac{x^2}{(x|x)}+C_{\g^\natural}=2x^2+C_{\g^\natural}.
$$
By Lemma 5.1 of \cite{KW}, we see that the eigenvalue for the action of $C_{\g_0}$ on $\g_{-1/2}$ is $h^\vee-1$. On the other hand, since $2ad(x)^2(v)=\half v$ for all $v\in\g_{-1/2}$, we see that the eigenvalue of $C_{\g_0}$ is also $\half+(\mu|\mu+2\rho_0)$, hence
$$
(\mu|\mu+2\rho_0)=h^\vee-\frac{3}{2}.
$$
 
Plugging
\eqref{delignehzero} into \eqref{1}Ê we obtain that  $k = -\frac{ h^{\vee} -1 }{2}$. $\g=sl(3)$ has to be dropped from the list because in that case $k=-1$ is a collapsing level.
\par

A case-by-case check shows that in cases $\g=spo(n|m)$ ($m\ge0$)  we have that $h_{0}^\vee= h^\vee -1$.
Plugging the data into \eqref{1}Ê we obtain  $k = -\frac{2}{3} h^{\vee}$. The cases $\g=spo(n|n+2)$ ($n\ge 2$) have to be discarded because $k+h^\vee=0$, while the cases $\g=spo(n|n-1)$ with $n>2$ have to be discarded because $k_1+h^\vee_{0}=k+\half(h^\vee+h_{0}^\vee)=0$. Finally the cases $spo(n|n-4)$ ($n\ge 4$) have to be discarded because $k=-2$ is a collapsing level.
\end{proof}

\begin{prop}\label{withcenter} Assume that $\g^\natural=\g^\natural_0\oplus \g^\natural_1$ with $\g^\natural_0\simeq\C$ and $\g^\natural_1$ simple.

If $\g=sl(m|m-3)$ with $m\ge 4$, then there are no non--collapsing conformal levels.

 In  other cases the non--collapsing conformal levels are 
\begin{enumerate}
\item $k=-\frac{2}{3}h^\vee$ if $\g=sl(m|m+1)$ ($m\ge2$), and $\g=sl(m|m-1)$ ($m\ge 3$);
\vskip 2pt
\item $k=-\frac{2}{3}h^\vee$ and $k=-\frac{h^\vee-1}{2}$ in all other cases.
\end{enumerate}
 \end{prop}
\begin{proof}
Write $\g^\natural=\C c\oplus \g^\natural_1$ with $c$ central in $\g^\natural$ and $\g^\natural_1$ simple. Since $\theta(c)=0$ we see that $c$ acts nontrivially on $\g_{\pm1/2}$, otherwise $c$ would be central in $\g$. We choose $c$ so that one of the eigenvalues of the action of $c$ on  $\g_{-1/2}$ is $1$. Since $\g_{-1/2}$ is self-dual as $\g^\natural$--module, we see that also $-1$ is an eigenvalue. Let $U^{\pm}$ be the corresponding eigenspaces. Since $U^{\pm}$ are $\g^\natural$--modules, we see that $\g_{-1/2}$ is the direct sum of at least two submodules. Browsing Tables 1--3, we see that $\g_{-1/2}$ breaks in at most two irreducible components,  so we have  
$\g_{-1/2}=U^+\oplus U^{-}$,
with $U^{\pm}$ irreducible $\g^\natural$--modules.
 
Let $V(\mu)$ be one of these two irreducible components.  According to \eqref{mui}, $V(\mu)=V_{\C c}(\mu^0)\otimes V_{\g^\natural_1}(\mu^1)$. We 
have to solve
\begin{equation}\label{eqconfcenter}
\frac{(\mu^0|\mu^0)}{2(k+h^\vee/2)}+\frac{(\mu^1|\mu^1+2\rho_0^1)}{2(k+\half(h^\vee+h^\vee_{0,1}))}=\frac{3}{2}.
\end{equation}

First we compute $(\mu^0|\mu^0)$: this is the eigenvalue for the action of $\frac{c^2}{(c|c)}$ on $U^+$, so $(\mu^0|\mu^0)=\frac{1}{(c|c)}$. On the other hand, letting $\kappa(\cdot,\cdot)$ denote the Killing form of $\g$, we have 
$$2h^\vee(c|c)=\kappa(c,c)=str(ad(c)^2)=2 \sdim \g_{-1/2}
$$
where, in the last equality, we used the fact that $c^2$ acts trivially on $\g_0+\g_{1}+\g_{-1}$ and as the identity on $\g_{\pm1/2}$. Using (5.6) of \cite{KW}, it follows that 
\begin{equation}\label{muzero}
(\mu^0|\mu^0)=\frac{h^\vee}{2h^\vee-4}.
\end{equation}
Arguing as in the proof of Proposition \ref{conformalweightsimple}, we see that 
$
(\mu^1|\mu^1+2\rho_0^1)+(\mu^0|\mu^0)=h^\vee-\frac{3}{2},
$
hence 
$$
(\mu^1|\mu^1+2\rho_0^1)=h^\vee-\frac{3}{2}-\frac{h^\vee}{2h^\vee-4}.
$$
Finally  a case-by-case check shows that $h^\vee_{0,1}=h^\vee-2$. Plugging these data in \eqref{eqconfcenter}, we obtain the result. The cases $\g=sl(m|m+1)$ with $k=-\frac{h^\vee-1}{2}$ and $m\ge2$ have to be discarded because $k+h^\vee=0$. The cases $\g=sl(m|m-1)$ with $k=-\frac{h^\vee-1}{2}$ and $m\ge3$ have to be discarded because $k_1+h^\vee_{0,1}=0$. The cases $\g=sl(m|m-3)$ with $k=-\frac{h^\vee-1}{2}$ and $m\ge4$ have to be discarded because $k$ is a collapsing level. The cases $\g=sl(m|m-3)$ with $k=-\frac{2}{3}h^\vee$ and $m\ge4$ have to be discarded because $k_1+h^\vee_{0,1}=0$.
\end{proof}

\begin{prop} Assume that $\g^\natural=\sum_{i=1}^r\g^\natural_i$ with $\g^\natural_1\simeq sl(2)$ and $r\ge 2$.

If $\g=osp(n+5|n)$ with $n\ge 2$ or $\g=D(2,1;a)$ with $a=\half,-\half,-\frac{3}{2}$, then there are no non--collapsing conformal levels. 

In the other cases the non--collapsing conformal levels are 
 \begin{enumerate}
 \item $k=-\frac{h^\vee-1}{2}$ if $\g=D(2,1;a)$ ($a\not\in\{\half,-\half,-\frac{3}{2}\}$), $\g=osp(n+8|n)$ ($n\ge0$), $\g=osp(n+2|n)$ ($n\ge2$), $\g=osp(n-4|n)$ ($n\ge8$);
 \vskip 3pt
\item $k=-\frac{2}{3}h^\vee$ if $\g=osp(n+7|n)$ ($n\ge0$),
$\g=osp(n+1|n)$ ($n\ge4$);
\vskip 3pt
\item $k=-\frac{2}{3}h^\vee$ and $k=-\frac{h^\vee-1}{2}$ in all other cases.
 \end{enumerate}
 \end{prop}
\begin{proof}
Let $\a$ be the positive root for $\g^\natural_1$. By browsing Tables 1--3  we see that $\g_{-1/2}$ is irreducible and that, if we write $\g_{-1/2}\simeq V_{sl(2)}(\mu^1)\otimes(\otimes_{i\ge 2} V_{\g^\natural_i}(\mu^i)$, then $\mu^1=\a/2$, where $\a$ is the positive root of $sl(2)$.

Then \eqref{tremezzi} becomes
\begin{equation}\label{tremezzisl2}
\frac{3(\a|\a)}{8(k+\frac{1}{2}(h^\vee+(\a|\a)))}+\sum_{i=2}^r\frac{(\mu^i|\mu^i+2\rho^i_0)}{2(k+\half(h^\vee+h^\vee_{0,i}))}=\frac{3}{2}.
\end{equation}

Let us assume first that $(\a|\a)=2$ and that $r=2$. This assumption leaves out only $D(2,1;a)$ and $D_4$. Arguing as in Proposition \ref{conformalweightsimple}, we see that 
$$\frac{3(\a|\a)}{4}+(\mu^2|\mu^2+2\rho^2_0)=\frac{3}{2}+(\mu^2|\mu^2+2\rho^2_0)=h^\vee-\frac{3}{2}
$$
so
$
(\mu^2|\mu^2+2\rho^2_0)=h^\vee-3.
$
Moreover one checks case-by-case that $h_{0,2}=h^\vee-4$. Plugging these data in \eqref{tremezzisl2}, we obtain that the solutions are $k=-\frac{h^\vee-1}{2}$ and  $k=-\frac{2}{3}h^\vee$. The case $\g=osp(n+1|n)$ with $n\ge4$ and $k=-\frac{h^\vee-1}{2}$ has to be discarded because $k+h^\vee=0$. This happens also in case $\g=osp(n+2|n)$ with $n\ge2$ and $k=-\frac{2}{3}h^\vee$. In case $\g=osp(n+7|n)$ with $n\ge0$ the level  $k=-\frac{h^\vee-1}{2}$ is collapsing. This also happens at level $k=-\frac{2}{3}h^\vee$ when $\g=osp(n-4|n)$ with $n\ge8$ or $\g=osp(n+5|n)$ with $n\ge2$. The case $\g=osp(n+5|n)$ with $n\ge2$ and $k=-\frac{h^\vee-1}{2}$ has to be dropped because $k_i+h^\vee_{0,i}=0$ for some $i$. The same happens in case $\g=osp(n+8|n)$ with $n\ge2$ and $k=-\frac{2}{3}h^\vee$.

Let us now discuss $\g=D(2,1;a)$: in this case $\g^\natural_2\simeq sl(2)$. Let $\beta$ be the positive root of $\g^\natural_2$. From Table 2 we see that $\g_{-1/2}=V_{sl(2)}(\a/2)\otimes V_{sl(2)}(\be/2)$. Then \eqref{tremezzi} becomes
\begin{equation}\label{tremezzisl3}
\frac{3(\a|\a)}{8(k+\frac{1}{2}(h^\vee+(\a|\a)))}+\frac{3(\be|\be)}{8(k+\frac{1}{2}(h^\vee+(\be|\be)))}=\frac{3}{2}.
\end{equation}
Plugging $h^\vee=0$, $(\a|\a)=-2(a+1)$, $(\be|\be)=2a$ into \eqref{tremezzisl3} and solving for $k$, one gets that $k=0=-\frac{2}{3}h^\vee$ or $k=\half=-\frac{h^\vee-1}{2}$. The value $k=0$ has to be discarded because this level is critical. If $a=\half$ or $a=-\frac{3}{2}$ then $k=\half$ is a collapsing level. If $a=-\half$ and $k=\half$ then $k_i+h^\vee_{0,i}=0$ for all $i$.

Finally, if $\g$ is of Type $D_4$ ($=osp(8|0)$), then $r=3$ and $\g^\natural_i\simeq sl(2)$ for all $i$. Furthermore, if $\a_i$ is the positive root of $\g^\natural_i$, then $\g_{-1/2}=\otimes_{i=1}^3 V_{sl(2)}(\a_i/2)$. Since $(\a_i|\a_i)=2$ for all $i$, it follows that \eqref{tremezzi} becomes
\begin{equation}\label{tremezzisl4}
\frac{9}{4(k+\frac{1}{2}(h^\vee+2))}=\frac{3}{2}.
\end{equation}
Since $h^\vee=6$, we obtain that $k=-\frac{h^\vee-1}{2}$ is the only solution of \eqref{tremezzisl4}. One checks directly that $k+h^\vee\ne0$, $k$ is a non--collapsing level, and that $k_i+h^\vee_{0,i}\ne0$ for all $i$.
\end{proof}



\begin{cor} \label{conformal-levels-formula}
$\mathcal{V}_{k}(\g^\natural)$ embeds conformally in ${W}_{k}(\g, \theta)$ if and only if $c(\g,k)=c_{sug}$.
\end{cor}
\begin{proof}
The set of collapsing levels $k$, such that $k+h^\vee\ne 0$ and $k_i+h^\vee_{0,i}\ne0$ whenever $k_i\ne 0$,   together with the set of conformal non--collapsing levels turns out to be exactly the set of solutions of the equation $c(\g,k)=c_{sug}$.
\end{proof}
\begin{rem} It is worthwile to note that, in general, equality of central charges does not imply conformal embedding. A counterexample is the following: take $W=\C\vac \subset V=W_{-1/3}(sl(3))$. Then $c(sl(3),-1/3)=0$, but, since $-1/3$ is a non-collapsing level, $\omega_V\ne 0$.
\end{rem}
\begin{example} \label{spec-1} Consider the case of $\g=sl(n)$, $n\ge3$.

If $n\ge4$ then $\g^\natural=\g^\natural_0\oplus \g^\natural_1$ with $\g^\natural_0=\C c$ the center of $\g^\natural$ and $\g^\natural_1\simeq sl(n-2)$. Thus $\g^\natural\simeq gl(n-2)$. In this case $h^\vee=n$, $h^\vee_{0,0}=0$, $h^\vee_{0,1}=n-2$.
Combining Theorem \ref{TT} and Proposition \ref{withcenter} we obtain 
\begin{enumerate}
\item If $k=-1$ then ${W}_{k}(sl(n),\theta)$ is the Heisenberg vertex algebra $V_{\frac{n-2}{2}}(\C c)$;
(this result  is also independently obtained in \cite{AM2} by using singular vectors from $V^{-1} (sl(n)) $ ).
\item If $k= -\frac{h^{\vee} }{2} =- \frac{n}{2}$ then ${W}_{k}(sl(n),\theta) \cong V_{k+1}  (sl(n-2) )$;
\item Assume that $k  \notin \{ -1, - h^{\vee}  /2 \}$. 
The embedding $   \mathcal V_{k}(gl(n-2))  \hookrightarrow {W}_{k}(sl(n),\theta)  $ is conformal if and only if
$k \in \{ \frac{1-n }{2}, - \frac{2}{3} n\}.$ 
\end{enumerate}
The equation $c(\g,k)=c_{sug}$ in this case reads
$$
\begin{cases}
\frac{k(n^2-1)}{k+n}-6k+n-4=1+\frac{(k+1)((n-2)^2-1)}{k+n-1}&\text{if $k\ne -\frac{n}{2}$}\\
\frac{k(n^2-1)}{k+n}-6k+n-4=\frac{(k+1)((n-2)^2-1)}{k+n-1}&\text{if $k= -\frac{n}{2}$}
\end{cases}
$$
whose solutions are precisely $-1,-\frac{2}{3}n ,\frac{1-n }{2},-\frac{n}{2}$.

The case $n=3$ is of particular interest. In this case $\g^\natural$ is one dimensional abelian: $\g^\natural=\C c $, $h^\vee=3$, $h^\vee_{0,0}=0$.
We have that ${ W}^ {k}(sl(3),\theta)$ is   the Bershadsky-Polyakov  vertex algebra, freely   generated by even vectors  $G^\pm$   of conformal weight $3/2$, the Virasoro vector $\omega$, and the Heisenberg vector $J=J^{\{c\}}$ (cf. \cite{Ber}, \cite{Ara-CMP}). 
By Theorem \ref{TT}, we obtain  ${ W}_{k}(sl(3),\theta)=\C\vac$ if $k=-\frac{h^\vee}{2}=-\frac{3}{2}$ while it is the Heisenberg vertex algebra $V_{\frac{1}{2}}(\C c)$ if $k =-1$.
 Since $k=-1=-\frac{h^\vee-1}{2}$, by  the proof   of Proposition \ref{conformalweightsimple}, there are no other conformal levels, but  condition \eqref{conf-level} of Proposition \ref{T} still holds. 
 Since $k$ is a collapsing level, we have that $ G^{ \{ u \} } = 0 $ in $W_k(\g)$.  But Proposition \ref{T}  gives that $ G^{ \{ u \} } \ne  0$ in ${\mathcal W}^{-1}(sl(3),\theta)$, 
  hence  the vertex algebra ${\mathcal W}^{-1}(sl(3),\theta)$ is intermediate between  ${W}^ {-1}(sl(3),\theta)$ and  ${W}_{-1}(sl(3),\theta)=V_{\frac{1}{2}}(\C c)$.
 
By \eqref{GGsimplified} with $k=-1$ we get
$$  {G^+} _{(0) }G^{-}   = \nu (\omega-\omega_{sug}) \qquad (\nu \ne 0). $$ This fact and \eqref{non-triv} in Proposition \ref{T} imply that ${\mathcal W}^{-1}(sl(3),\theta)$ is an extension of a Heisenberg vertex algebra by elements $G^\pm$ with $\l$-brackets
$$ [ {G^+} _{\lambda} G^{+} ] = [ {G^+} _{\lambda} G^{-} ]  = [ {G^-} _{\lambda} G^{-} ]  = 0. $$

One can also show that for every $k \in {{ \Z}_{>  0} }$, $(G^+ _{(-1)}   )^ k {\bf 1}$ is a non-trivial singular vector in $  {\mathcal W}^{-1}(sl(3),\theta)$ and therefore $ {\mathcal W}^{-1}(sl(3),\th)$ contains infinitely many non-trivial ideals.

Another occurrence of this phenomenon happens when 
$\g=sp(4)$ at level $k=-2$. In this case $\g^\natural\simeq sl(2)$, the level $k$ is collapsing but condition \eqref{tremezzi} still holds.  We have that $\mathcal W^ {-2} (\g,\theta)$ is a non-simple extension of $\mathcal V_{-2}(sl(2))$ with
 $[{G ^{\{ u \} }} _{\lambda}  G ^{\{ v \} } ] = 0$
for every $u, v \in {\g}_{-1/2} $.
It is shown in \cite{AKMMP2} that also relation $:G ^{\{ u \} }  G ^{\{ v \}}:=0$ holds.
As in the $\g=sl(3)$ example above, one can construct infinitely many ideals in the vertex algebra $\mathcal W^ {-2} (sp(4),\theta)$.


\end{example}

\section{Affine vertex algebra $V_{k}(sl(n))$ and $W$-algebras}\label{sln+1inW}
 In this Section we apply our simplified formula \eqref{GGsimplified} for $[{G^{\{u\}}}_\lambda{G^{\{v\}}}]$ and the conformal embedding  of $\mathcal V_{k}(gl(n))$ in ${W}_{k}(sl(2|n), \theta)$ at level $k=-\frac{h^\vee-1}{2}=\frac{n-1}{2}$ to construct a realization of $V_{-\frac{n+1}{2}}(sl(n+1))$ in the tensor product of $W_k(sl(2|n),\theta)$ with a lattice vertex algebra for $n\ge4$, $n\ne5$: see Theorem \ref{sln+1}.

 Since $W_k(psl(2|2),\theta)$ is the $N=4$ superconformal algebra,  our result can be considered a higher rank generalization of the realization, given in \cite{A-2014}, of $V_{-\frac{3}{2}}(sl(3))$  as a vertex subalgebra of  the tensor product of a $N=4$ superconformal algebra with a lattice vertex algebra.

Let $\g=sl(2|n)$ and $k=-\frac{h^\vee-1}{2}=\frac{n-1}{2}$.  In this case $\g^\natural \simeq gl(n)$, $k_{0}=k+\half h^\vee=\frac{1}{2}$ and $k_{1}=k+\half(h^\vee-h^\vee_{0,1})=\frac{n+1}{2}$. By Proposition \ref{withcenter},
 $\mathcal V_{\frac{n-1}{2}}(\g^\natural)$ embeds conformally in  ${W}_{k}(\g,\theta)$. 
 This embedding is a generalization of the correspondence between affine vertex algebra $\mathcal V_{k}(sl(2))$ and superconformal vertex algebras from \cite{A-1999} and \cite{FST} (see also \cite{A-2007} for the critical level version of this correspondence).



First we start with the following useful criterion for establishing that some elements of $W^k(sl(2|n),\theta)$ are in the ideal  $\mathcal{J}^{k} (sl(2 \vert n))$ (cf. \eqref{J}).
\begin{lemma} \label{kriterij}Assume that $v \in W^{k} (sl(2 \vert n),\theta)$ satisfies
$$(\omega_{sug} )_{0 }v = a v, \quad {\omega}_0 v = b v, \quad a\ne b,  $$
where $a, b \in {\mathbb C}$. Then $v \in \mathcal{J}^{k} (sl(2 \vert n))$.
\end{lemma}
\begin{proof}
Recall that $\omega_{sug}=\omega$ in $\mathcal W^{k} (sl(2 \vert n),\theta)= W^{k} (sl(2 \vert n),\theta)/ \mathcal{J}^{k} (sl(2 \vert n))$. If $v\ne0$ in $\mathcal W^{k} (sl(2 \vert n),\theta)$ then it  would be an eigenvector  for both $(\omega_{sug} )_{0 }$ and ${\omega}_0$, but then the corresponding eigenvalues must coincide.
\end{proof}
\vskip 10pt
It follows that in order to construct elements of the ideal $\mathcal{J}^{k} (sl(2 \vert n))$ it is enough to construct common eigenvectors $v$ for $\omega_0$ and $(\omega_{sug})_0$ which have different eigenvalues. 

To simplify notation let us introduce
$$
G^+_i=G^{\{e_{2,2+i}\}},\quad  G^-_i=G^{\{e_{2+i,1}\}}.
$$
We shall consider the elements
$$ v^{+}_{i,j}  = :G_i ^{+}  G_j ^{+} :,\ \ v^{-}_{i,j}  = :G_i ^{-}  G_j ^{-} :\quad (i,j =1, \dots, n). $$

We shall now give some structural formulas. Recall from the proof of Proposition \ref{withcenter} that $\g^\natural=\g^\natural_0\oplus\g^\natural_1$ with $\g^\natural_0=\C c$ and $\g^\natural_1\simeq sl(n)$, $\g_{-1/2}=U^+\oplus U^-$, and  $(c|c)=2-\frac{4}{h^\vee}$. Furthermore, note that $span(e_{2,2+i}\mid 1\le i\le n)$ and $span(e_{2+i,1}\mid 1\le i\le n)$ are both stable under the action of $c$ and $sl(n)$. It follows that we can choose $c$ so that $U^+=span(e_{2,2+i}\mid 1\le i\le n)$ and $U^-=span(e_{2+i,1}\mid 1\le i\le n)$. With this setting we have

\begin{lemma} \label{L-str-1}
\bea
&&{ G_i ^{ + } }_{(s) }G_j ^{+ } = { G_i ^{ - } }_{(s) }G_j ^{- }=0\quad (s\ge0), \label{prima}\\
&& :G_i ^{ + }  G_i ^{+ }:=:G_i ^{ - }  G_i ^{- }: = 0 , \label{primabis}\\
&&{ G_i ^{ +} }_{(2)} G_j ^{-}= \delta_{i,j} 2 (k+1) ( k+ \frac{h^{\vee}}{2}) {\bf 1},\label{seconda} \\
&&  {G_i ^{ +} }_{(1) }G_j ^{-}=    2 (k + \frac{h^{\vee} } {2}  ) J ^{ \{e_{2+j, 2+ i} \} }  \quad (i \ne j), \label{terza} \\
&&   {G_i ^{ +} }_{(1)} G_i ^{-} =  \frac{ n-2 }{n }   ( k  + 1 )J ^{ \{ c \} }  +  2 ( k+\frac{h^{\vee} } {2} ) J^{\{a\}}.  \label{quarta}
\eea
where  $a=(e_{1,1}+e_{2+i,2+i})^\natural_1=(e_{2,2}+e_{2+i,2+i})^\natural_1 \in \g^\natural_1\simeq sl(n)$.\end{lemma}
 \begin{proof} Formula \eqref{prima} follows by  observing that, when $u=e_{2\,2+i},v=e_{2\,2+j}$ or 
 $u=e_{2+i\,1},v=e_{2+j\,1}$, all brackets in the r.h.s. of \eqref{GGsimplified} vanish. 
  
Recall that, since $\g_{-1/2}$ is purely odd in this case,
 $$
 :{G_i ^{ \pm } } G_j ^{\pm } : =- :G_j ^{ \pm }G_i ^{\pm }:+\int_{-\partial}^0[{G_i^\pm}_\mu G^\pm_j]d\mu.
 $$
 By \eqref{prima}, $[{G^+_i}_\mu G^+_j]=[{G^-_i}\mu G^-_j]=0$, so \eqref{primabis} follows.

 Formulas \eqref{seconda},  \eqref{terza}, and \eqref{quarta} follow from  Lemma \ref{formG} and the explicit expression for $p(k)$ given in Table 4.
 \end{proof}

\begin{proposition}  \label{rel-1}
Assume that $1 \le i, j \le n$, $i \ne j$. Then we have
\begin{enumerate}
\item \label{vijeigen} $$ ( {\omega_{ sug}})_0 v^{\pm} _{i,j} =  \frac{4 (n-2)}{n-1}  v^{\pm} _{i,j}. $$
\item\label{vijzero} Assume that $n \ne 5$. Then 
$ v^{\pm} _{i,j } \in \mathcal{J}^{k} (sl(2 \vert n)) $,
hence
\begin{equation} \label{normalorderzero}:G_i ^{+}  G_j ^{+}:=  :G_i ^{-}  G_j ^{-}:=0 \quad \mbox{in} \   {W}_{k}(sl(2|n),\theta). 
\end{equation}
\end{enumerate}
\end{proposition}
\begin{proof}
Set
$$ \widetilde{U^{\pm}} ^{(2)} = \mbox{span}_{\mathbb C} (v^{\pm} _{i,j} \mid 1 \le i < j \le n).$$
Formula \eqref{JG} shows that ${J^{\{a\}}}_{(s)} u=0$ for all $u\in\widetilde{U^{\pm}} ^{(2)}$ and $s>0$.  As a $gl(n)$--module  $\widetilde{U^+}^{(2)}$ is 
isomorphic to $\bigwedge^2 \mathbb C^n$ and  $\widetilde{U^-}^{(2)}$  is 
isomorphic to $\bigwedge^2( {\mathbb C}^n)^*$. It follows that the eigenvalue of  $( {\omega_{ sug}})_0$ on  $\widetilde{U^+}^{(2)}$ is the same as the eigenvalue of $( {\omega_{ sug}})_0$ on $\widetilde{U^-}^{(2)}$ and the latter is
$$\frac{(2\mu_0|2\mu_0)}{2(k+\half h^\vee)}+\frac{(\omega_2|\omega_2+2\rho_1)}{2(k+\half(h^\vee+h^\vee_{0,1})}.
$$
Using \eqref{muzero} and substituting $k=\frac{n-1}{2}$, $h^\vee=2-n$, and $h^\vee_{0,1}=-n$, we find assertion \eqref{vijeigen}.

 Since
$\omega_0 v^{\pm} _{i,j} = 3 v^{\pm} _{i,j}$,
 Lemma \ref{kriterij} implies that, if $n \ne 5$, then $v^{\pm} _{i,j} \in  \mathcal{J}^{k} (sl(2 \vert n))$  hence \eqref{vijzero} follows.
\end{proof}




From now on  we assume that $n \ne  5$. Recall that we assumed $n\ge4$.
We have chosen  $c$ so that ${J^{\{c\}}}_{(0)} G^{\pm}_i=\pm G^{\pm}_i$. It follows that we can define a $\Z$--gradation on ${W}_{k}(sl(2|n),\theta)$
$$ {W}_{k}(sl(2|n),\theta) = \bigoplus {W}_{k} (sl(2|n),\theta)^{(i) },
$$
setting
$$
{W}_{k}(sl(2|n),\theta)^{(i) }= \{ v \in {W}_{k} (sl(2|n),\theta)\ \vert \ J^{ \{ c \} } _{(0)} v = i v \}. $$
Clearly $\mathcal V_k(sl(2|n)^\natural)  \subset  {W}_{k} (sl(2|n),\theta)^{(0)}$, and $G_i ^{\pm} \in{W}_{k}(sl(2|n),\theta)^{(\pm 1)}. $

Let $F_{-1}$ be  the simple lattice
 vertex superalgebra $V_{\Z \varphi}$ associated to the lattice $\Z \varphi$ where
$\langle \varphi, \varphi \rangle =-1$. As a vector space,
 $F_{-1} =   M_{\varphi} (1) \otimes {\mathbb C}[\Z \varphi] $ where $M_{\varphi} (1)$
 is the Heisenberg vertex algebra generated
 by the field $\varphi(z)= \sum_{n \in {\Z} } \varphi_{(n)} z ^{-n-1}$, and ${\mathbb C}[\Z \varphi] $
  is the group algebra with generator
 $ e^{\varphi}$. $F_{-1}$ admits the following $\Z$--gradation:
 $$F_{-1} = \bigoplus_{m \in \Z} F_{-1} ^{(m)}, \quad F_{-1} ^{(m)} := M_{\varphi} (1) \otimes e^{m \varphi}. $$

Consider now the simple vertex superalgebra $ {W}_{k}(sl(2|n),\theta) \otimes F_{-1}$. Set
$$ \overline{\varphi} = J ^{\{c\} } \otimes  {\bf 1}  +   {\bf 1}\otimes \varphi_{(-1)} {\bf 1}. $$
Note that $ \overline{\varphi}_{(0)}$ acts semisimply on $ {W}_{k}(sl(2|n),\theta) \otimes F_{-1}$.  We consider the following vertex subalgebra
\begin{equation}\label{degreezero}  ({W}_{k}(sl(2|n),\theta) \otimes F_{-1} ) ^{(0)} =\mbox{Ker} \ \overline{\varphi}_{(0)}.
\end{equation}
Explicitly,
$$
({W}_{k}(sl(2|n),\theta) \otimes F_{-1} ) ^{(0)}=\bigoplus_{i \in {\Z} } {W}_{k}(sl(2|n),\theta)^{(i)} \otimes F_{-1} ^ {(i)}.
$$
Define
$$ \overline{U} ^{\pm} = \mbox{span}_{\mathbb C} \{ G_i ^{\pm} \otimes e ^{\pm \varphi} \ \vert \ 1 \le i \le n \}, $$
$$ \overline{U} ^0 = \{J^{\{a\}}\otimes {\bf 1} \vert a \in sl(n) \}\oplus {\mathbb C} J ^{\{c \} } \otimes {\bf 1} \oplus {\mathbb C} {\bf 1}\otimes\varphi_{(-1)} {\bf 1}. $$
Clearly, $ \overline{U} ^0 $ generates the vertex subalgebra  of $ ({W}_{k}(sl(2|n),\theta) \otimes F_{-1} ) ^{(0)}$ which is  equal to  $ \mathcal{V}_k(sl(2|n)^\natural) \otimes  M_{\varphi} (1). $

\begin{proposition} \label{str}\ %

\begin{enumerate}
\item The vertex algebra $({W}_{k}(sl(2|n),\theta) \otimes F_{-1} ) ^{(0)}$ is simple.
\item\label{generateu} The vertex algebra $({W}_{k}(sl(2|n),\theta) \otimes F_{-1} ) ^{(0)}$ is generated by the vector space
 $\overline{U} ^- \oplus \overline{U} ^ 0  \oplus \overline{U} ^ + .$
\end{enumerate}
\end{proposition}
\begin{proof}
The proof is similar to that of Theorem 10.1. in \cite{A-2014}. 

It is a general fact that the degree zero component of a $\ganz$-graded simple vertex algebra is simple. The simplicity of $({W}_{k}(sl(2|n),\theta) \otimes F_{-1} ) ^{(0)}$ follows.

Let $S$ be the vertex subalgebra of $({W}_{k}(sl(2|n),\theta) \otimes F_{-1} ) ^{(0)}$  generated by the space $\overline{U} ^- \oplus \overline{U} ^ 0  \oplus \overline{U} ^ +$.   To prove \eqref{generateu}, we show that ${W}_{k}(sl(2|n),\theta)^ {(\ell)} \otimes F_{-1}  ^ {(\ell)}\subset S$ 
for every $\ell \in {\Z}$.
Let $\mathcal U^\pm= \mbox{span}_{\mathbb C} \{ G_i ^{\pm}  \vert \ 1 \le i \le n \}$. Note that, since $\omega=\omega_{sug}$,  ${W}_{k}(sl(2|n) ^ {(\ell)} \otimes F_{-1}  ^ {(\ell)} $ is linearly spanned by the vectors of the form
\bea
\label{general-form}
&& {G^{+ } _{i_1} }_{(n_1)} \cdots {G^{+} _{i_r} }_{(n_r)} {G^{ - }_{j_1} }_{ (m_1)} \cdots {G^{-}_{j_s}}_{   (m_s)}w \otimes w^ {(1)} ,  
\eea
where  $\ell =r-s$, 
 $w\in  \mathcal{V}(sl(2|n)^\natural)  $, $  w^ {(1)}  \in F_{-1} ^ {(\ell)} $,
 thus 
 ${W}_{k}(sl(2|n) )^ {(\ell)} \otimes F_{-1}  ^ {(\ell)} $ is contained in the linear span of 
 $$
\underbrace{ \mathcal U^+\cdot\mathcal U^+\cdots \mathcal U^+}_{\text{$r$ times}}\cdot\underbrace{\mathcal U^-\cdot\mathcal U^-\cdots\mathcal U^-}_{\text{$s$ times}}\cdot \mathcal V_k(sl(2|n)^\natural)\cdot F^{(\ell)}_{-1}.
 $$
 By using relations
$$ G_i ^{\pm} \otimes \vac = (G_i ^{\pm} \otimes e^{\pm \varphi} ) _{(-2)} (  {\bf 1} \otimes e^{\mp \varphi} ) \quad (1\le i \le n)$$
one sees that  $\mathcal U^\pm\subset \overline U^\pm\cdot F^{(\mp1)}_{-1}$ so  
 ${W}_{k}(sl(2|n),\theta)^ {(\ell)} \otimes F_{-1}^ {(\ell)} $ is contained in the linear span of 
\begin{align*}
&\underbrace{ \overline U^+\cdot\overline U^+\cdots \overline U^+}_{\text{$r$ times}}\cdot\underbrace{\overline U^-\cdot\overline U^-\cdots\overline U^-}_{\text{$s$ times}}\cdot \mathcal V_k(sl(2|n)^\natural)\cdot F^{(s-r)}_{-1}\cdot F^{(\ell)}_{-1}\\
&\subseteq\underbrace{ \overline U^+\cdot\overline U^+\cdots \overline U^+}_{\text{$r$ times}}\cdot\underbrace{\overline U^-\cdot\overline U^-\cdots\overline U^-}_{\text{$s$ times}}\cdot \mathcal V_k(sl(2|n)^\natural)\cdot F^{(0)}_{-1}
\end{align*}

The claim follows.
\end{proof}

\begin{lemma} \label{bracket} 
We have:
\begin{enumerate}
\item\label{i} $({G_i^+\otimes e^{+\varphi}})_{(r)}(G_j^+\otimes e^{+\varphi})= ({G_i^-\otimes e^{-\varphi}})_{(r)}(G_j^-\otimes e^{-\varphi})=0 $ for $r \ge 0$;
\vskip1pt
\item\label{ii}$({G_i^+\otimes e^{+\varphi}})_{(0)}({G_j^-\otimes e^{-\varphi}})=-J^{\{e_{2+j,2+i}\}}\otimes\vac$ if $i\ne j$;
\vskip1pt
\item\label{iii}$({G_i^+\otimes e^{+\varphi}})_{(0)}({G_i^-\otimes e^{-\varphi}})=-\frac{(n-2)(n+1)}{2n}J^{\{{c}\}}-\frac{n+1}{2}\vac\otimes\varphi_{(-1)}\vac-J^{\{a\}}$ with  $a=(e_{1,1}+e_{2+i,2+i})^\natural_1=(e_{2,2}+e_{2+i,2+i})^\natural_1 \in \g^\natural_1\simeq sl(n)$;
\vskip1pt
\item\label{iv}$({G_i^+\otimes e^{+\varphi}})_{(1)}({G_j^-\otimes e^{-\varphi}}) =-\delta_{i,j}\frac{n+1}{2}\vac\otimes\vac$;
\vskip1pt
\item\label{v} $({G_i^+\otimes e^{+\varphi}})_{(m)}({G_j^-\otimes e^{-\varphi}} )=0$ for $m\geq 2$.
\end{enumerate}
\end{lemma}
\begin{proof}Since, by \cite[(4.3.2)]{KacV}, we have 
$$
({G_i^+\otimes e^{\varphi}})_{(r)}(G^+_j\otimes e^{+\varphi})=-\sum_{m\ge r-1}{G_i^+}_{(m)}G_j^+\otimes {e^{\varphi}}_{(-m+r-1)}e^{\varphi},
$$
and, likewise,
$$
({G_i^-\otimes e^{-\varphi}})_{(r)}(G^-_j\otimes e^{-\varphi})=-\sum_{m\ge r-1}{G_i^-}_{(m)}G_j^-\otimes {e^{-\varphi}}_{(-m+r-1)}e^{-\varphi},
$$
 assertion \eqref{i} follows from relations \eqref{normalorderzero},  \eqref{prima}, and \eqref{primabis}.
 
Since
 $$
({G_i^+\otimes e^{\varphi}})_{(r)}(G^-_j\otimes e^{-\varphi})=-\sum_{m\ge r+1}{G_i^+}_{(m)}G_j^-\otimes {e^{\varphi}}_{(-m+r-1)}e^{-\varphi},
$$
formulas \eqref{ii}--\eqref{v} follow form  Lemma \ref{L-str-1} and  the fact that ${e^{\varphi}}_{(-2)}e^{-\varphi}=\vac$ while ${e^{\varphi}}_{(-3)}e^{-\varphi}=\varphi_{(-1)}\vac$.
\end{proof}
With notation as in Section \ref{uno}, we set $V^{k'}(sl(n+1))=V^{B}(sl(n+1))$ where $B=k'(\cdot,\cdot)$ and  $(\cdot,\cdot)$ is the  invariant form of $sl(n+1)$, normalized by the condition $(\a,\a)=2$ for any root $\a$.
We also let, if $k'+n+1\ne0$, $ V_{k'}(sl(n+1))$  be its unique simple quotient.
\begin{theorem}\label{sln+1} We have
$$({W}_{k}(sl(2|n),\theta) \otimes F_{-1} ) ^{(0)}  \cong V_ {-\frac{n+1}{2}} (sl(n+ 1)) \otimes M_{\overline{\varphi} } (1),$$
where 	
$ M_{\overline{\varphi} }(1)$ is the Heisenberg vertex algebra generated by
$ \overline{\varphi} $.
\end{theorem}

\begin{proof}
Set 
\bea  \overline{U} ^{0 , 0} & =  & \{ v \in   \overline{U} ^{0 } \ \vert \ \overline{\varphi}_{(n)} v = 0  \quad \mbox{for} \ n \ge 0 \}\nonumber \\
& =  & \{ {J^{\{a\}}} \otimes {\bf 1} \vert a \in \g^\natural_1 \} \oplus {\mathbb C} w, \nonumber \eea
where
$ w =   J ^{\{c\} }\otimes   {\bf 1}  +   \frac{n}{n-2}  {\bf 1}  \otimes\varphi_{(-1)} {\bf 1}. $

Set $$U^{ext} = \overline{U} ^{ +} \oplus  \overline{U} ^{0,0 } \oplus \overline{U} ^{ -}.$$
Identify  $\g^\natural_1$ with $sl(n)$ by identifying $a\in sl(n)$ with $\begin{pmatrix}0&0\\0&a\end{pmatrix}$ in $sl(2|n)$. Embed $gl(n)$ in $sl(n+1)$ by mapping $a$ to $\iota(a)=\begin{pmatrix}  -tr(a) & 0 \\ 0 & a\end{pmatrix}$. Define  $\gamma: sl(n+1)\to U^{ext} $ by mapping:
$$
\iota(a) \mapsto J^{\{a\}}\otimes{\bf 1},\,\,a\in sl(n),\ \ \iota(I_n)\mapsto \tfrac{(n+1)(n-2)}{2} w,$$
$$e_{1\,1+i}\mapsto  G_i^+\otimes e^{\varphi},\quad e_{1+i \,1}\mapsto  G_i^-\otimes e^{-\varphi},\,1\leq i\leq n,$$
We claim    that, if  $a,  b\in sl(n+1)$, then 
\begin{equation}\label{xly}
[\gamma(a)_\lambda \gamma(b)]=\gamma([a,b])-\frac{n+1}{2}\lambda (a,b),\end{equation}
Indeed if $a,b\in sl(n)$, \eqref{xly} follows from \eqref{JJ} observing that $(a|b)=-(a,b)$. Next, assume $a\in sl(n),\, b= e_{1\,1+i}$. Then, by \eqref{JG}, we have 
$$[J^{\{a\}}\otimes {\bf 1}_{\lambda} G^+_i\otimes e^\varphi]=[{J^{\{a\}}}_{\lambda} G^+_i]\otimes e^\varphi=G^{\{[a,e_{2\,2+i}]\}}\otimes e^\varphi,
$$
which is  \eqref{xly}  in the present case, since $G^{\{[a,e_{2\,2+i}]\}}\otimes e^\varphi=\gamma([\iota(a),e_{1\,1+i}])$. The same argument proves \eqref{xly} when $a\in sl(n)$ and $b=e_{1+i\,1}$.
If $a=e_{1\,1+i},\, b= e_{1\,1+j}$ or $b=e_{1+j\,1}$, then \eqref{xly} follows from Lemma \ref{bracket}. 
If $a=I_n$ and $b\in sl(n)$ then 
$[\gamma(\iota(I_n))_\lambda \gamma(\iota(b))]=\tfrac{(n+1)(n-2)}{2} [w_\lambda (J^{\{b\}}\otimes\vac)]=\tfrac{(n+1)(n-2)}{2} [{J^{\{c\}}}_\lambda J^{\{b\}}]\otimes\vac=0$ as desired. 
If $a=I_n$ and $b=e_{1\,1+i}$ then 
\begin{align*}[\gamma(\iota(I_n))_\lambda \gamma(b)]&=\tfrac{(n+1)(n-2)}{2} [({J^{\{c\}}}_\lambda G^+_i)\otimes e^\varphi)]+\tfrac{(n+1)n}{2} [G^+_i\otimes [\varphi_\lambda e^\varphi)]\\
&=-(n+1) G^+_i\otimes e^\varphi=\gamma([I_n,e_{1\,1+i}]).
\end{align*}
The case with $a=I_n$ and $b=e_{1+i\,1}$ is treated similarly. Finally, if $a=b=I_n$, then
$$
[\gamma(\iota(I_n))_\lambda \gamma(\iota(I_n))]=(\tfrac{(n+1)(n-2)}{2} )^2([{(J^{\{c\}}}_\lambda J^{\{c\}})\otimes \vac]+(\tfrac{n}{n-2})^2 [\vac\otimes [\varphi_\lambda \varphi]).
$$
It follows that $\gamma(\iota(I_n))_{(0)} \gamma(\iota(I_n))=0$ and
$$
\gamma(\iota(I_n))_{(1)} \gamma(\iota(I_n))=-\tfrac{(n+1)^2}{2} n=-\tfrac{n+1}{2} (I_n,I_n),
$$
as desired. 
Having shown that \eqref{xly} holds, we furthemore observe that
\begin{align*}
\overline \varphi_{(1)}w=({J^{\{c\}}}_{(1)}J^{\{c\}})\otimes \vac+\tfrac{n}{n-2}\vac\otimes \varphi_{(1)}\varphi
=\half(c|c)-\tfrac{n}{n-2}=0.
\end{align*}
This, together with \eqref{degreezero}, shows that $[\overline{\varphi}_\lambda a]=0$ for all $a\in U^{ext}$.   It follows that there is a vertex algebra homomorphism 
$$ V^{-\frac{n+1}{2}} (sl(n+1)) \otimes  M_{\overline{\varphi} }(1) \rightarrow  ({W}_{k}(sl(2|n), \theta) \otimes F_{-1} ) ^{(0)} . $$

Proposition \ref{str} implies that this homomorphism is surjective. Since   $({W}_{k}(sl(2|n), \theta) \otimes F_{-1} ) ^{(0)} $ is a simple vertex algebra, we get that 
$$ V_{-\frac{n+1}{2}} (sl(n+1)) \otimes  M_{\overline{\varphi} }(1) \cong ({W}_{k}(sl(2|n), \theta) \otimes F_{-1} ) ^{(0)} .$$
\end{proof}

  \vskip5pt
  \footnotesize{
  \noindent{\bf D.A.}:  Department of Mathematics, University of Zagreb, Bijeni\v{c}ka 30, 10 000 Zagreb, Croatia;
{\tt adamovic@math.hr}
  
\noindent{\bf V.K.}: Department of Mathematics, MIT, 77
Mass. Ave, Cambridge, MA 02139;\newline
{\tt kac@math.mit.edu}

\noindent{\bf P.MF.}: Politecnico di Milano, Polo regionale di Como,
Via Valleggio 11, 22100 Como,
Italy; {\tt pierluigi.moseneder@polimi.it}

\noindent{\bf P.P.}: Dipartimento di Matematica, Sapienza Universit\`a di Roma, P.le A. Moro 2,
00185, Roma, Italy; {\tt papi@mat.uniroma1.it}

\noindent{\bf O.P.}:  Department of Mathematics, University of Zagreb, Bijeni\v{c}ka 30, 10 000 Zagreb, Croatia;
{\tt perse@math.hr}
}

\end{document}